\documentclass[12pt, a4paper,reqno]{amsart}
\usepackage{amscd, amssymb, pb-diagram}
\usepackage{amsthm}

\usepackage{tikz}
\usepackage[margin=2.0cm]{geometry}
\usepackage{color}
\usepackage{tensor}
\usetikzlibrary{arrows,matrix,backgrounds,decorations.markings}
\usepackage{empheq} 
\numberwithin{equation}{section}

\def\ker{\operatorname{Ker}}
\def\im{\operatorname{Im}}

\def\max{\operatorname{max}}

\def\Z{\mathbb{Z}}

\def\AA{\mathcal{A}}

\def\QQ{\mathcal{Q}}

\newcommand{\ep}{\varepsilon}
\newcommand{\fib}[4]{#1 \tensor[_{#2}]{\times}{_{#3}} #4}
\renewcommand{\phi}{\varphi}

\tikzstyle{vertex}=[circle]
\tikzstyle{goto}=[->,shorten >=1pt,>=stealth,semithick]

\newtheorem{thm}{Theorem}[section]

\newtheorem{lemma}[thm]{Lemma}

\theoremstyle{definition}
\newtheorem{definition}[thm]{Definition}

\theoremstyle{remark}
\newtheorem{remark}[thm]{Remark}

\newtheorem{example}[thm]{Example}

\newtheorem*{Acknowledgements}{Acknowledgements}

\begin{document}

\date{\today}
\title[Functorial properties of Putnam's homology theory for Smale spaces]{Functorial properties of Putnam's homology theory for Smale spaces}

\author{Robin J. Deeley}
\address{Robin J. Deeley,  Laboratorie de Mathematiques, Universite Blaise Pascal, Clermont-Ferrand II, France}
\email{robin.deeley@gmail.com}
\author[D. Brady Killough]{D. Brady Killough}
\address{D. Brady Killough, Mathematics, Physics and Engineering, Mount Royal University, Calgary, Alberta, Canada T3E 6K6}
\email{bkillough@mtroyal.ca}
\author[Michael F. Whittaker]{Michael F. Whittaker}
\address{Michael F. Whittaker, School of Mathematics and Applied Statistics, University of Wollongong, NSW 2522, Australia}
\email{mfwhittaker@gmail.com}

\thanks{This research was supported by the Natural Sciences and Engineering Research Council of Canada through a postdoctoral fellowship and the Australian Research Council (DP1096001).}


\begin{abstract}
We investigate functorial properties of Putnam's homology theory for Smale spaces. Our analysis shows that the addition of a conjugacy condition is necessary to ensure functoriality. 
Several examples are discussed that elucidate the need for our additional hypotheses. Our second main result is a natural generalization of Putnam's Pullback Lemma from shifts of finite type to non-wandering Smale spaces. 
\end{abstract}

\maketitle

\section{Introduction}

In this paper we consider functorial properties of the homology theory for Smale spaces introduced by Putnam in \cite{put}. The fundamental tool used to define Putnam's homology theory is an s/u-bijective pair. Putnam proves that the homology of a Smale space is independent of the choice of an s/u-bijective pair, provided a pair exists. One of Putnam's seminal results establishes that s/u-bijective pairs exist for every non-wandering Smale space \cite[Theorem 2.6.3]{put}, generalizing a celebrated result of R. Bowen \cite{Bow}. Putnam asserts that the homology theory is a covariant functor with respect to s-bijective maps and a  contravariant functor for u-bijective maps \cite[Theorem 5.4.1 and Theorem 5.4.2]{put}. 

Unfortunately, there is a minor mistake in the statement of Theorem 5.4.1 in \cite{put}: the fibre product maps are only assumed to be onto, but must be conjugacies in order that the proof in \cite{put} is valid. We discuss this issue in detail in Section \ref{modified3511}. In particular, we would like to emphasize that the s/u-bijective pair constructed in \cite[Theorem 5.4.2]{put} and the one used in \cite[Chapter 6]{put} (see Remark \ref{LF_rem}) satisfy our more stringent conjugacy condition. As such, this issue causes no problems with any of the other results in \cite{put}.

Our first main result is to show that Putnam's functoriality results hold under a more stringent (than considered in \cite[Theorem 5.4.1]{put}) class of s/u-bijective pairs that are used to define the induced map on homology from an s-bijective or u-bijective map; the proof is exactly the one given in \cite{put}. However, these results require us to prove that the definition of the induced map is natural with respect to the isomorphism considered in the proof of the independence of the homology on the choice of s/u-bijective pair (see \cite[Section 5.5]{put}). Subtleties arising from the interaction of s-bijective and u-bijective maps (and their induced maps on homology) are a prevailing theme of the proofs of this section and in fact the entire paper.

The second main result is a natural generalization of Theorem 3.5.11 in \cite{put} from the case of shifts of finite type with the dimension group to the case of Smale spaces with Putnam's homology theory. Since Putnam's homology is a generalization of the dimension group, the statement of this result is natural enough; the proof, on the other hand, is rather involved.

One might hope for even more general versions of these results. In Section \ref{examples}, we provide explicit examples showing that the most obvious generalizations of our results do not hold. Interestingly, we were able to construct all these examples using shifts of finite type. The heart of the matter seems to be the failure of \cite[Theorem 3.5.11]{put} to hold without adding assumptions beyond the commutativity of the diagram in the statement of the theorem. Despite all these negative results, Theorem \ref{ontoCloseOneToOne} implies that the (seemingly strong) hypotheses we use hold rather generally. All in all, Theorem 3.5.11 of \cite{put} and possible generalizations of it (along with the failure of possible generalizations) are the joint starting point for both our main results.

Many constructions involving Putnam's homology theory begin by fixing a particular choice of s/u-bijective pair, showing that a desired result holds for that particular pair, and then showing that it is independent of the choice of s/u-bijective pair or that (at least) there is a particular nice class of s/u-bijective pairs for which the result is independent of the choice. The proofs of our main results follow this general framework.

We now give the section by section content of this paper. Section \ref{preliminaries} introduces the fundamental definitions and results required in the sequel. In particular, we introduce Smale spaces and their basic properties with specific focus on maps between Smale spaces, and then introduce Putnam's homology theory for Smale spaces. Our first main result appears in Section \ref{modified3511} where we use \cite[Theorem 3.5.11]{put} to prove a corrected version of Putnam's \cite[Theorem 5.4.1]{put}, showing that the homology theory is functorial. Along the way we provide a discussion of the possibility of weakening the hypotheses required in the theorems. On the one hand, we show that certain natural ones fail in one way or another. On the other hand, Theorem \ref{ontoCloseOneToOne} implies that the more stringent hypotheses holds in rather general situations. In Section \ref{naturality_homology} we show that the homology theory is natural with respect to the choice of s/u-bijective pair used to define the homology groups under our more stringent hypothesis. In Section \ref{main_theorem_section} we show that the natural generalization of \cite[Theorem 3.5.11]{put} extends to general nonwandering Smale spaces. The final section provides examples showing that the hypotheses in our theorems are necessary.

\begin{Acknowledgements}
We thank Ian Putnam for many interesting and useful discussions about the content of this paper, and for his guidance and support during this early part of our academic careers. We also thank the referee for a number of useful suggestions.

The authors are also grateful to the Courant Research Centre at Georg-August-University G\"{o}ttingen, the Fields Institute, the University of Victoria, and the University of Wollongong for facilitating this collaboration by providing funding for research and conference visits.
\end{Acknowledgements}

\section{Preliminaries}\label{preliminaries}

The content of this section provides the essential results required later in the paper. We make no attempt to put the results in context or expand on detail. We refer the reader to the appropriate results in \cite{put} throughout the paper for a detailed treatment. For this reason the reader is advised to have a copy of Putnam's \emph{A Homology Theory for Smale Spaces} \cite{put} handy.

The first part is dedicated to Smale spaces and their properties. In the second part we discuss maps between Smale spaces culminating in Putnam's definition of an s/u-bijective pair associated to a Smale space. Putnam uses an s/u-bijective pair on a Smale space $(X,\varphi)$ to generalize Bowen's seminal theorem: given any irreducible Smale space $(X,\varphi)$, there is a shift of finite type $(\Sigma,\sigma)$ and a finite-to-1 factor map $\pi:(\Sigma,\sigma) \to(X,\varphi)$. An s/u-bijective pair for $(X,\varphi)$ consists of a Smale space $(Y,\psi)$ with totally disconnected unstable sets and a Smale space $(Z,\zeta)$ with totally disconnected stable sets such that the fibre product of $(Y,\psi)$ and $(Z,\zeta)$ is a shift of finite type that recovers Bowen's Theorem for $(X,\varphi)$.  In the final subsection, we summarize Putnam's homology theory. Putnam's key idea is to pass Krieger's dimension group invariant on a shift of finite type to a non-wandering Smale space using an s/u-bijective pair. The theory is reminiscent of \v{C}ech cohomology where the s/u-bijective pair plays the role of the ``good cover"; we refer the reader to the introduction of Putnam's manuscript \cite{put} for further insight into his approach.

\subsection{Smale spaces}\label{Sec:Smale}

A Smale space $(X,\varphi)$ consists of a compact metric $X$ space along with a homeomorphism $\varphi:X \rightarrow X$ such that every point $x \in X$ has a neighbourhood that is the product of two local coordinates, one that contracts under the action $\varphi$ and the other that contracts under the action of $\varphi^{-1}$ (expands under $\varphi$). The precise definition requires the definition of a bracket map satisfying certain axioms as follows. 

\begin{definition}[{\cite[p.19]{put}, \cite{Rue1}}]
A Smale space $(X,\varphi)$ consists of a compact metric space $X$ with metric $d$ along with a homeomorphism $\varphi: X \to X$ such that there exist constants $ \ep_{X} > 0, 0<\lambda < 1$ and a continuous bracket map
\[ (x,y) \in X, d(x,y) \leq \ep_{X} \mapsto [x, y] \in X \]
satisfying the bracket axioms:
\begin{itemize}
\item[B1] $\left[ x, x \right] = x$,
\item[B2] $\left[ x, [ y, z] \right] = [ x, z]$,
\item[B3] $\left[ [ x, y], z \right] = [ x,z ]$, and
\item[B4] $\varphi[x, y] = [ \varphi(x), \varphi(y)]$;
\end{itemize}
for any $x, y, z$ in $X$ when both sides are defined. In addition, $(X,\varphi)$ is required to satisfy the contraction axioms:
\begin{itemize}
\item[C1] For $y,z \in X$ such that $[x,y]=y$, we have $d(\varphi(y),\varphi(z)) \leq \lambda d(y,z)$ and
\item[C2] For $y,z \in X$ such that $[x,y]=x$, we have $d(\varphi^{-1}(y),\varphi^{-1}(z)) \leq d(y,z)$.
\end{itemize}
\end{definition}

For each $x$ in $X$ and $ 0 < \ep \leq \ep_{X}$, there are local stable and unstable sets defined by
\begin{align*}
 X^{s}(x, \ep) & = \{ y \in X \mid d(x,y) \leq \ep, [y,x] =x \}, \\ 
X^{u}(x, \ep) & = \{ y \in X \mid d(x,y) \leq \ep, [x,y] =x \}.
\end{align*}
The following diagram illustrates the bracket with respect to these sets. 
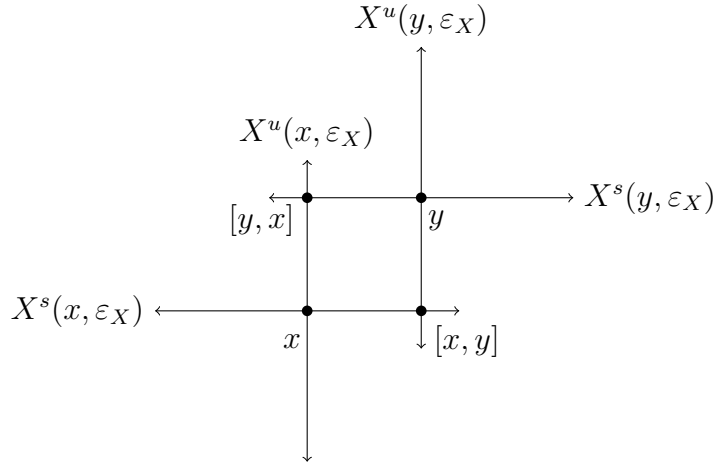
\begin{center}
\begin{tikzpicture}
\tikzstyle{axes}=[]
\begin{scope}[style=axes]
	\draw[<->] (-3,-1) node[left] {$X^s(x,\ep_X)$} -- (1,-1);
	\draw[<->] (-1,-3) -- (-1,1) node[above] {$X^u(x,\ep_X)$};
	\node at (-1.2,-1.4) {$x$};
	\node at (1.1,-1.4) {$[x,y]$};
	\pgfpathcircle{\pgfpoint{-1cm}{-1cm}} {2pt};
	\pgfpathcircle{\pgfpoint{0.5cm}{-1cm}} {2pt};
	\pgfusepath{fill}
\end{scope}
\begin{scope}[style=axes]
	\draw[<->] (-1.5,0.5) -- (2.5,0.5) node[right] {$X^s(y,\ep_X)$};
	\draw[<->] (0.5,-1.5) -- (0.5,2.5) node[above] {$X^u(y,\ep_X)$};
	\node at (0.7,0.2) {$y$};
	\node at (-1.6,0.2) {$[y,x]$};
	\pgfpathcircle{\pgfpoint{0.5cm}{0.5cm}} {2pt};
	\pgfpathcircle{\pgfpoint{-1cm}{0.5cm}} {2pt};
	\pgfusepath{fill}
\end{scope}
\end{tikzpicture}
\end{center}
The bracket then encodes the local product structure as follows: if $d(x,y) < \ep_X$, then $\{[x,y]\} = X^s(x,\ep_X) \cap X^u(y,\ep_X)$. A dynamical system $(X,\varphi)$ with a bracket map is a {\em Smale space}. We note that if a bracket map exists on $(X,\varphi)$ then it is unique. We are interested in Smale spaces satisfying various topological recurrence conditions - namely {\em non-wandering}, {\em irreducible}, and {\em mixing}, see \cite[Definitions 2.1.3, 2.1.4, and 2.1.5]{put} for the precise definitions.

Suppose $(X,\varphi)$ is a Smale space and $x \in X$, there are global stable and unstable equivalence relations on $X$ given by
\begin{eqnarray}
\nonumber
X^s(x) & = & \{y \in X | \lim_{n \rightarrow +\infty} d(\varphi^n(x),\varphi^n(y)) = 0\}, \\
\nonumber
X^u(x) & = & \{y \in X | \lim_{n \rightarrow +\infty} d(\varphi^{-n}(x),\varphi^{-n}(y)) = 0\}.
\end{eqnarray}
We will also denote stable (respectively, unstable) equivalence of points by $x \sim_s y$ ($x \sim_u y)$. Our notation indicates a connection between the global stable and local stable set of a point. Indeed, for any $x$ in $X$ and $\ep >0$, we have $X^s(x,\ep) \subset X^s(x)$. Further, a point $y$ is in $X^s(x)$ if and only if there exists $N \geq 0$ such that $\varphi^n(y)$ is in $X^s(\varphi^n(x),\ep)$ for all $n \geq N$, see \cite[Proposition 2.1.11]{put}. An important consequence of the local stable sets being subsets of the global stable sets is that the local stable sets $X^s(x,\ep)$ form a neighbourhood base for a locally compact and Hausdorff topology on the global stable set $X^s(y)$ as $0 <\ep< \ep_X$ and $x \in X^s(y)$ vary \cite[Proposition 2.1.12]{put}. The same results also hold for the unstable sets.

The building blocks of Putnam's homology theory are the shifts of finite type. These are precisely the zero dimensional Smale spaces and they come equipped with a homology theory called Krieger's dimension group. We briefly recount \cite[Section 2.2]{put} where the dimension group is defined in a way that suits our purpose.

Suppose $G=(G^0,G^1,i,t)$ is a strongly connected finite directed graph (there is a path of edges between every pair of vertices) with vertices $G^0$, edges $G^1$, and each edge $e \in G^1$ is given by a directed edge from vertex $i(e)$ to vertex $t(e)$, see \cite[Definition 2.2.1]{put}. For $K > 0$, the paths of length $K$ in $G$ are $K$-tuples $G^K:=\{(e^1,\cdots,e^K) \mid t(e^i)=i(e^{i+1}) \text{ for } 1 \leq i \leq K-1\}$ and for $K=0$ we define the paths of length zero to be the elements of $G^0$. For every $K \geq 1$ a new graph, denoted $G(K)$, is defined from $G$ with vertices $G(K)^0=G^{K-1}$, edges $G(K)^1=G^K$, with $i,t:G^K \to G^{K-1}$ defined by $i(e^1,\cdots,e^K)=(e^1,\cdots,e^{K-1})$ and $t(e^1,\cdots,e^K)=(e^2,\cdots,e^{K})$. Note that for $1 \leq L \leq K$, iterating the maps $i,t:G^K \to G^{K-1}$ $L$-times we obtain maps $i^L,t^L:G^K \to G^{K-L}$.

To any strongly connected directed graph there is an associated shift of finite type \cite[Definition 2.1.1]{LM}. Suppose $G$ is a strongly connected graph, a shift of finite type $(\Sigma_G,\sigma)$ is obtained taking the compact Hausdorff space $\Sigma_G$ consisting of all bi-infinite paths $(e^k)_{k \in \Z}$ in $G$ and the homeomorphism $\sigma: \Sigma_G \to \Sigma_G$ given by the left shift map $\sigma(e)^k=e^{k+1}$ for all $e \in G$. If $e \in \Sigma_G$ and $0\leq K \leq L$, we define $e^{[K,L]}$ to be the tuple $(e^K,\cdots,e^L)$ and we also define $e^{[K+1,K]}=t(e^K)=i(K+1)$. The bracket map and metric on $(\Sigma_G,\sigma)$ are given in \cite[Definition 2.2.5]{put}. Stable equivalence is right tail equivalence and unstable equivalence is left tail equivalence \cite[Definition 2.2.6]{put}. We note that a Smale space $(X,\varphi)$ is totally disconnected if and only if $(X,\varphi)$ is conjugate to a shift of finite type \cite[Theorem 2.2.8]{put}.

\subsection{Maps on Smale spaces}

Suppose $(X,\varphi)$ and $(Y,\psi)$ are dynamical systems, a map $\pi:(Y,\psi) \to (X,\varphi)$ is a continuous function $\pi:Y \to X$ such that $\pi \circ \psi=\varphi \circ \pi$. A surjective map is called a factor map. Note that maps between Smale spaces are automatically compatible with the bracket map \cite[Theorem 2.3.2]{put}.

We now restrict our attention to maps on Smale spaces. We specify our attention to the stable sets and note that each property has an analogous property with respect to the unstable sets. If $(X,\varphi)$ and $(Y,\psi)$ are Smale spaces and $\pi:(Y,\psi) \to (X,\varphi)$ is a map, then for any $y \in Y$ a routine argument with the bracket map shows that $\pi(Y^s(y)) \subset X^s(\pi(y))$. A map $\pi:(Y,\psi) \to (X,\varphi)$ to be \emph{s-resolving} if for any $y \in Y$ the restriction $\pi|_{Y^s(y)}:Y^s(y) \to X^s(\pi(y))$ is injective \cite{Fri}. Resolving maps have extremely nice properties as described in \cite[Section 2.5]{put}, however Putnam's homology theory requires an even stronger condition on maps. An s-resolving map is called \emph{s-bijective} if for all $y \in Y$ the restriction $\pi|_{Y^s(y)}:Y^s(y) \to X^s(\pi(y))$ is bijective. The importance of s-bijective maps is described in \cite[Theorem 2.5.6]{put}: if $\pi:(Y,\psi) \to (X,\varphi)$ is s-bijective, then $(\pi(Y),\varphi|_{\pi(Y)})$ is a Smale space. Putnam proves \cite[Theorem 2.5.8]{put} that if $\pi:(Y,\psi) \to (X,\varphi)$ is an s-resolving map and $(Y,\psi)$ is non-wandering, then $\pi$ is s-bijective. 

We now introduce one of the primary objects in Putnam's theory, fibre products of Smale spaces \cite[Definition 2.4.1]{put}. Suppose $(X,\varphi)$, $(Y_1,\psi_1)$, and $(Y_2,\psi_2)$ are dynamical systems and $\pi_i:(Y_i,\psi_i) \to (X,\varphi)$ are maps for $i=1,2$, then the fibre product $(\fib{Y_1}{\pi_1}{\pi_2}{Y_2},\psi_1 \times \psi_2)$ consists of the compact Hausdorff space
\[
\fib{Y_1}{\pi_1}{\pi_2}{Y_2} :=\{(y_1,y_2) \in Y_1 \times Y_2 \mid \pi_1(y_1)=\pi_2(y_2)\},
\]
along with a homeomorphism $\psi_1 \times \psi_2$. As shown in \cite[Theorem 2.4.1]{put}, if $(X,\varphi)$, $(Y_1,\psi_1)$, and $(Y_2,\psi_2)$ are Smale spaces, then the fibre product $(\fib{Y_1}{\pi_1}{\pi_2}{Y_2},\psi_1 \times \psi_2)$ is also a Smale space with metric $d((y_1,y_2),(y_1',y_2'))=\max\{d(y_1,y_1'),d(y_2,y_2')\}$.  The construction of a fibre product can be iterated, and we will be interested in $N$-fold fibre products over a single space. Suppose $\pi:(Y,\psi) \to (X, \varphi)$ is a map and define $(Y_N(\pi),\psi)$ to be the fibre product
\[
Y_N(\pi) :=\{(y_0,y_1, \cdots, y_N) \in Y^N \mid \pi(y_i)=\pi(y_j) \text{ for all } 0\leq i,j \leq N\},
\]
where $\psi(y_0,y_1,\dots,y_N)=(\psi(y_0),\psi(y_1),\dots,\psi(y_N))$. The dynamical system $(Y_N(\pi),\psi)$ is a Smale space whenever $(X,\varphi)$ and $(Y,\psi)$ are Smale spaces \cite[Proposition 2.4.4]{put}.

We will need two key results about maps on fibre products. The first is \cite[Theorem 2.5.13]{put}: suppose $(\fib{Y_1}{\pi_1}{\pi_2}{Y_2},\psi_1 \times \psi_2)$ is a fibre product of Smale spaces and let $P_i:\fib{Y_1}{\pi_1}{\pi_2}{Y_2} \to (Y_i,\psi_i)$ be the projection maps for $i=1,2$, then if $\pi_1$ is s-bijective, so is $P_2$. For the second result, we need a preliminary definition. Suppose $(Y_N(\pi), \psi)$ is an $N$-fold fibre product of a Smale space $(Y,\psi)$ over $(X,\varphi)$, then there is a map $\delta_n:Y_N(\pi) \to Y_{N-1}(\pi)$ given by
\begin{equation}\label{delete_n_Cech}
\delta_n(y_0,y_1,\cdots, y_N) = (y_0,y_1,\cdots,\check{y_n},\cdots,y_N),
\end{equation}
where $\check{y}_n$ denotes deleting the $n$th coordinate. Putnam's second key result \cite[Theorem 2.5.14]{put} shows that if $\pi:(Y,\psi) \to (X,\varphi)$ is s-bijective, then so is $\delta_n:(Y_N(\pi),\psi) \to (Y_{N-1}(\pi),\psi)$ for all $N \geq 1$ and $0\leq n \leq N$.

One of the seminal results for Smale spaces is Bowen's Theorem \cite{Bow}: If $(X,\varphi)$ is a non-wandering Smale space, then there exists a shift of finite type $(\Sigma,\sigma)$ and a factor map $\pi:(\Sigma,\sigma) \to (X,\varphi)$ such that $\pi$ is finite-to-1 and 1-to-1 on a dense $G_\delta$ subset of $\Sigma$. We now discuss Putnam's strengthening of Bowen's Theorem. We begin with Putnam's definition of an s/u-bijective pair.

\begin{definition}[{\cite[Definition 2.6.2]{put}}]
Suppose $(X,\varphi)$, $(Y,\psi)$, and $(Z,\zeta)$ are Smale spaces. The tuple $\pi=(Y,\psi,\pi_s,Z,\zeta,\pi_u)$ is called an s/u-bijective pair if
\begin{enumerate}
\item $\pi_s:(Y,\psi) \to (X,\varphi)$ is an s-bijective factor map,
\item $Y^u(y)$ is totally disconnected for all $y \in Y$,
\item $\pi_u:(Z,\zeta) \to (X,\varphi)$ is a u-bijective factor map, and
\item $Z^s(z)$ is totally disconnected for all $z \in Z$.
\end{enumerate}
\end{definition}

Putnam's generalization of Bowen's Theorem is the following.

\begin{thm}[{\cite[Theorem 2.6.3]{put}}]
If $(X,\varphi)$ is a non wandering Smale space, then there exists an s/u-bijective pair for $(X,\varphi)$.
\end{thm}

Suppose $(Y,\psi,\pi_s,Z,\zeta,\pi_u)$ is an s/u-bijective pair for $(X,\varphi)$. For each $L,M \geq 0$,
\[
\Sigma_{L,M}(\pi):=\{(y_0,\cdots,y_L,z_0,\cdots,z_M) \mid y_l \in Y, z_m \in Z, \pi_s(y_l)=\pi_u(z_m), 0 \leq l \leq L, 0 \leq m \leq M\}.
\]
By definition, $\Sigma_{0,0}(\pi)=\Sigma(\pi)$ is the fibre product of $(Y,\psi)$ and $(Z,\zeta)$. Define a map $\sigma: \Sigma_{L,M} \to \Sigma_{L,M}$ by
\[ 
\sigma_{L,M}(y_0,\cdots,y_L,z_0,\cdots,z_M)=(\psi(y_0),\cdots,\psi(y_L),\zeta(z_0),\cdots,\zeta(z_M)).
\]
For $L,M \geq 1$, $0 \leq l \leq L$, and $0 \leq m \leq M$, let $\delta_{l,\,}:\Sigma_{L,M} \to \Sigma_{L-1,M}$ and $\delta_{\,,m}:\Sigma_{L,M} \to \Sigma_{L,M-1}$ be the maps
\begin{align*}
\delta_{l,\,}(y_0,\cdots,y_L,z_0,\cdots,z_M)&=(y_0,\cdots,\check{y_l},\cdots,y_L,z_0,\cdots,z_M), \text{ and} \\
\delta_{\,,m}(y_0,\cdots,y_L,z_0,\cdots,z_M)&=(y_0,\cdots,y_L,z_0,\cdots,\check{z_m},\cdots,z_M), \\
\end{align*}
where $\check{y}_l$ and $\check{z}_m$ again denotes deleting the indicated coordinate. Putnam shows \cite[Theorem 2.6.6]{put} that if $\pi$ is an s/u-bijective pair for $(X,\varphi)$, then $(\Sigma_{L,M}(\pi),\sigma)$ is a shift of finite type for all $L,M \geq 0$. Moreover, \cite[Theorem 2.6.13]{put} shows that $\delta_{l,\,}:\Sigma_{L,M} \to \Sigma_{L-1,M}$ is an s-bijective factor map and $\delta_{\,,m}:\Sigma_{L,M} \to \Sigma_{L,M-1}$ is a u-bijective factor map.

\subsection{Putnam's homology theory on Smale spaces}

In this section we recall the construction of Putnam's homology theory. We begin with Krieger's dimension groups and show how the maps from the last section give rise to a double complex associated to a Smale space $(X,\varphi)$. Again the reader should be familiar with Putnam's manuscript since we are merely giving an overview.

Suppose $(\Sigma,\sigma)$ is a shift of finite type. In \cite{Kri}, Krieger associates two abelian groups to $(\Sigma,\sigma)$ known as Krieger's dimension group, denoted $D^s(\Sigma,\sigma)$ and $D^u(\Sigma,\sigma)$. We present a description of the stable dimension group and note that a similar construction gives the unstable version. Define 
 \[
 CO^s(\Sigma,\sigma):=\{E \in \Sigma^s(e) \mid e \in \Sigma \text{ and } E \text{ is a nonempty clopen subset of } \Sigma^s(e)\},
 \]
and for $E,F \in CO^s(\Sigma,\sigma)$ such that $[E,F]=E$ and $[F,E]=F$ let $\sim$ be the smallest equivalence relation such that $E \sim F$ if and only if $\sigma(E)=\sigma(F)$. Let $[E]$ denote the equivalence class of an element $E \in CO^s(\Sigma,\sigma)$. We are now ready to define the stable dimension group $D^s(\Sigma,\sigma)$.

\begin{definition}[\cite{Kri}, {\cite[Definition 3.3.2]{put}}]
Suppose $(\Sigma,\sigma)$ is a shift of finite type. Define $D^s(\Sigma,\sigma)$ to be the free abelian group on the $\sim$-equivalence classes of $CO^s(\Sigma,\sigma)$ modulo the subgroup generated by $[E \cup F]-[E]-[F]$, where $E,F,E \cup F \in CO^s(\Sigma,\sigma)$ and $E \cap F = \varnothing$.
\end{definition}
 
Krieger's dimension groups are functorial in the following ways.

\begin{thm}[{\cite[Theorem 3.4.1 and 3.5.1]{put}}]\label{Krieger_functorial}
Suppose $\pi:(\Sigma,\sigma) \to (\Sigma',\sigma)$ and $\pi':(\Sigma',\sigma) \to (\Sigma'',\sigma)$ are factor maps between shifts of finite type:
\begin{enumerate}
\item If $\pi$ and $\pi'$ are s-bijective, then $\pi^s([E])=[\pi(E)]$ induces a well-defined group homomorphism from $D^s(\Sigma,\sigma)$ to $D^s(\Sigma',\sigma)$ and $(\pi' \circ \pi)^s=\pi'^s \circ \pi^s$;
\item If $\pi$ and $\pi'$ are u-bijective, then $\pi^u([E])=[\pi(E)]$ induces a well-defined group homomorphism from $D^u(\Sigma,\sigma)$ to $D^u(\Sigma',\sigma)$ and $(\pi' \circ \pi)^u=\pi'^u \circ \pi^u$;
\item If $\pi$ and $\pi'$ are u-bijective, $E' \in CO^s(\Sigma',\sigma)$, and $E_1,\cdots,E_L \in CO^s(\Sigma,\sigma)$ satisfy conditions (1) and (2) of \cite[Theorem 3.5.1]{put}, then $\pi^{s*}([E'])=\sum_{l=1}^L[E_l]$ induces a well-defined group homomorphism from $D^s(\Sigma',\sigma)$ to $D^s(\Sigma,\sigma)$ and $(\pi \circ \pi')^{s*}=\pi'^{s*} \circ \pi^{s*}$; and
\item If $\pi$ and $\pi'$ are s-bijective, $E' \in CO^u(\Sigma',\sigma)$, and $E_1,\cdots,E_L \in CO^u(\Sigma,\sigma)$ satisfy analogues of (1) and (2) of \cite[Theorem 3.5.1]{put}, then $\pi^{u*}([E'])=\sum_{l=1}^L[E_l]$ induces a well-defined group homomorphism from $D^u(\Sigma',\sigma)$ to $D^u(\Sigma,\sigma)$ and $(\pi \circ \pi')^{u*}=\pi'^{u*} \circ \pi^{u*}$.
\end{enumerate}
\end{thm}

Putnam's homology theory is a generalization of Krieger's invariant to Smale spaces possessing an s/u-bijective pair. From this point forward fix a Smale space $(X,\varphi)$ and assume that $\pi$ is an s/u-bijective pair for $(X,\varphi)$. We begin by recalling Putnam's \cite[Definition 5.1.1]{put}: For each $L,M \geq 0$, define $C^s(\pi)_{L,M}=D^s(\Sigma_{L,M}(\pi),\sigma)$ and for either $L<0$ or $M<0$ define $C^s(\pi)_{L,M}=0$. Let $d^s(\pi)_{L,M}:C^s(\pi)_{L,M} \to C^s(\pi)_{L-1,M} \oplus C^s(\pi)_{L,M+1}$ be the map defined by
\[
d^s(\pi)_{L,M}=\sum_{0\leq l\leq L} (-1)^l \delta_{l,\,}^s + \sum_{0\leq m\leq M+1} (-1)^{L+m} \delta_{\,,m}^{s *}.
\]
There is a similar construction for the unstable sets, see \cite[Definition 5.1.1~(2)]{put}.
For all $L,M \geq 0$ there is an action of the permutation groups $S_{L+1} \times S_{M+1}$ on $\Sigma_{L,M}(\pi)$ that commutes with the dynamics. As in \cite[Definition 5.1.5]{put}, for $L,M \geq 0$ taking appropriate quotients and images of $D^s(\Sigma_{L,M}(\pi)$ with respect to the permutation group $S_{L+1} \times S_{M+1}$ leads to a double complex denoted $D_{\QQ,\AA}^s(\Sigma_{L,M}(\pi))$ associated with an s/u-bijective pair $\pi$ of $(X,\varphi)$. Putnam then defines $C_{\QQ,\AA}^s(\pi)_{L,M}=D_{\QQ,\AA}^s(\Sigma_{L,M}(\pi))$ and lets $d_{\QQ,\AA}^s(\pi)_{L,M}$ be the appropriate boundary map on $C_{\QQ,\AA}^s(\pi)_{L,M}$ induced from $d^s(\pi)_{L,M}$. The complex $C_{\QQ,\AA}^s(\pi)$ is shown to have only a finite number of nonzero entries in \cite[Theorem 5.1.10]{put}. We then have the following definition.

\begin{definition}[{\cite[Definition 5.1.11]{put}}]\label{HomDefn}
Suppose $\pi$ is an s/u-bijective pair for $(X,\varphi)$. Then $H^s_*(\pi)$ is the homology of the double complex $(C_{\QQ,\AA}^s(\pi),d_{\QQ,\AA}^s(\pi))$ given by
\[
H^s_N(\pi):=\ker\Big(\bigoplus_{L-M=N} d_{\QQ,\AA}^s(\pi)_{L,M}\Big) \Big/ \im\Big(\bigoplus_{L-M=N+1} d_{\QQ,\AA}^s(\pi)_{L,M}\Big).
\]
\end{definition}

A similar construction gives the homology $H^u_*(\pi)$. Putnam shows in \cite[Theorem 5.5.1]{put} that $H^s_*$ is independent of the s/u-bijective pair $\pi$. Thus defining a homology theory for $(X,\varphi)$ denoted $H^s_*(X,\varphi)$. 


\section{Modified Results from Putnam's Homology Theory}\label{modified3511}

In this section we restate several of Putnam's theorems from \cite{put} that the following sections will rely heavily on. The primary result of the section is an alteration of \cite[Theorem 5.4.1]{put}, where the two product maps are conjugacies rather than surjections. While this might seem to be a step in the wrong direction, the proof of \cite[Theorem 5.4.1]{put} relies on \cite[Theorem 3.5.11]{put} where the maps must be conjugacies. Example \ref{2to1example} shows that the conclusion of \cite[Theorem 3.5.11]{put} fails to hold if the product map is injective but not surjective or surjective but not injective.

\begin{thm}[{\cite[Theorem 3.5.11]{put}}]\label{mod3.5.11}
Suppose that 
\[
\begin{CD}
(\Sigma, \sigma) @>>\eta_1> (\Sigma_1, \sigma) \\
@VV \eta_2 V @VV \pi_1 V \\
(\Sigma_2, \sigma) @>> \pi_2 > (\Sigma_0, \sigma)\\
\end{CD}
\]
is a commutative diagram of non-wandering shifts of finite type in which $\eta_1$ and $\pi_2$ are s-bijective factor maps, and $\eta_2$ and $\pi_1$ are u-bijective factor maps. Moreover, suppose that $\eta_2 \times \eta_1 : (\Sigma, \sigma) \to \fib{(\Sigma_2, \sigma)}{\pi_2}{\pi_1}{(\Sigma_1, \sigma)}$ is a conjugacy.
Then
\[
\eta_1^{s} \circ \eta_2^{s*} = \pi_1^{s*} \circ \pi_2^{s}: D^s(\Sigma_2,\sigma) \to D^s(\Sigma_1,\sigma). 
\]
\end{thm}

This result underlies much of Putnam's work in \cite{put} as well as our main results in this paper. We now examine the conjugacy condition in the statement of Theorem \ref{mod3.5.11}. In particular, Theorem \ref{ontoCloseOneToOne} gives a rather weak and tractable condition for the map $\eta_2\times \eta_1$ to be a conjugacy (given that it is onto).

\begin{lemma}
Suppose $\pi_1: (X_0,\varphi_0) \to (X_1,\varphi_1)$ and $\pi_2: (X_1,\varphi_1) \to (X_2,\varphi_2)$ are maps between non-wandering Smale spaces.  Moreover, suppose both $\pi_2$ and $\pi_2 \circ \pi_1$ are s-bijective (respectively u-bijective) factor maps.  If $\pi_1$ is onto, then $\pi_1$ is also s-bijective (respectively u-bijective).
\end{lemma}

\begin{proof}
We prove the s-bijective case, the u-bijective case is analogous.  Fix $x_0 \in X_0$, we need to show that $\pi_1|_{X^s_0(x_0)}: X^s_0(x_0) \to X^s_1(\pi_1(x_0))$ is bijective.  

For surjectivity,  suppose $y_1 \in X^s_1(\pi_1(x_0))$, then $\pi_2(y_1) \in X^s_2(\pi_2 \circ \pi_1 (x_0))$ and since  $\pi_2 \circ \pi_1$ is s-bijective, there exists a unique $y_0 \in X^s_0(x_0)$ such that $\pi_2 \circ \pi_1 (y_0) = \pi_2(y_1)$. Now since $y_0 \sim_s x_0$, $\pi_1(y_0) \sim_s \pi_1(x_0) \sim_s y_1$.  Since $\pi_2$ is s-bijective, we have that $\pi_1(y_0) = y_1$ and hence $\pi_1|_{X^s_0(x_0)}$ is surjective.

For injectivity,  suppose $y_0, y'_0 \in X^s_0(x_0)$ and $\pi_1(y_0) = \pi_1(y'_0)$, then $\pi_2 \circ \pi_1 (y_0) = \pi_2 \circ \pi_1 (y'_0)$ and $\pi_2 \circ \pi_1$ is injective. Thus, $y_0 = y_0'$, implying that $\pi_1|_{X^s_0(x_0)}$ is injective.
\end{proof}

\begin{thm}\label{thm:n-to-one}
Suppose $\pi: (Y, \psi) \to (X, \varphi)$ is a factor maps between irreducible Smale spaces which is both s-bijective and u-bijective.  Then $\pi$ is $n$-to-$1$ for some positive integer $n$.
\end{thm}

\begin{proof}
Since $\pi$ is s-bijective, by \cite[Theorem 2.5.3]{put} there exists $M \geq 1$ such that for any $x \in X$, there exist $y_1, y_2, \ldots, y_K \in Y$ with $K \leq M$ such that $\pi^{-1}(X^u(x)) = \cup_{k=1}^K Y^u(y_k)$ where $Y^u(y_i) \cap Y^u(y_j)=\varnothing$ for $i \neq j$. Now since $\pi$ is u-bijective, $\pi|_{Y^u(y_k)}: Y^u(y_k) \to X^u(x)$ is a bijection. 
Thus, $|\pi^{-1}\{z\}| = K$ for all $z \in X^u(x)$.

A similar argument, reversing the roles of s-bijective and u-bijective, shows that for any $x' \in X$ there exists $L \leq M$ such that $|\pi^{-1}\{z\}| = L$ for all $z \in X^u(x')$.

Now suppose that $X$ is irreducible and let $x,x' \in X$. Using Smale's Spectral Decomposition Theorem \cite[Theorem 2.1.13]{put} (also see \cite[Section 7.4]{Rue1} or \cite[Theorem 6.2]{Sma}), which implies that an irreducible Smale space is composed of a finite number of mixing components that cyclically permute under the action of $\varphi$, there exists $l$ such that $\varphi^l(x)$ and $x'$ are in the same mixing component.  Therefore, there exists $z \in X^s(\varphi^l(x)) \cap X^u(x')$, so $|\pi^{-1}\{x\}| = |\pi^{-1}\{\varphi^l(x)\}| = |\pi^{-1}\{z\}| = |\pi^{-1}\{x'\}|$. In other words, $\pi$ is $n$-to-$1$ for some $n$.
\end{proof}

\begin{lemma}\label{Lem:n-1}
Suppose 
\[
\begin{CD}
(X,\varphi) @>>\eta_1> (X_1,\varphi_1) \\
@VV \eta_2 V @VV \pi_1 V \\
(X_2,\varphi_2) @>> \pi_2 > (X_0,\varphi_0) \\
\end{CD}
\]
is a commutative diagram of irreducible Smale spaces in which $\eta_1$ and $\pi_2$ are s-bijective factor maps, and $\eta_2$ and $\pi_1$ are u-bijective factor maps. Moreover, suppose 
$\eta_2 \times \eta_1 : (X, \varphi) \to \fib{(X_2, \varphi_2)}{\pi_2}{\pi_1}{(X_1, \varphi_1)}$ is onto.
Then $\eta_2 \times \eta_1$ is $n$-to-$1$ for some positive integer $n$.
\end{lemma}

\begin{proof}
Combining Theorem \ref{thm:n-to-one} and Lemma \ref{Lem:n-1} gives the result.
\end{proof}

Two obvious ways to relax the hypotheses in Theorem \ref{mod3.5.11} is to abandon the requirement that $\eta_2 \times \eta_1$ is surjective or to abandon the requirement that $\eta_2 \times \eta_1$ is injective.  Lemma \ref{Lem:n-1} shows that if we keep the requirement that $\eta_2 \times \eta_1$ is surjective, then the map is $n$-to-1. Example \ref{2to1example} shows that the surjectivity requirement is necessary, and that if $\eta_2 \times \eta_1$ is surjective but is $n$-to-1 for $n > 1$ the conclusion of Theorem \ref{mod3.5.11} fails to hold. 

Although the condition that $\eta_2 \times \eta_1$ is a conjugacy appears rather strong, the next theorem shows that it holds for a rather general class of maps $\eta_1$ and $\eta_2$.

\begin{thm}\label{ontoCloseOneToOne}
Suppose 
\[
\begin{CD}
(X,\varphi) @>>\eta_1> (X_1,\varphi_1) \\
@VV \eta_2 V @VV \pi_1 V \\
(X_2,\varphi_2) @>> \pi_2 > (X_0,\varphi_0) \\
\end{CD}
\]
is a commutative diagram of irreducible Smale spaces in which $\eta_1$ and $\pi_2$ are s-bijective factor maps, and $\eta_2$ and $\pi_1$ are u-bijective factor maps. Moreover, suppose $\eta_2 \times \eta_1 : (X, \varphi) \to \fib{(X_2, \varphi_2)}{\pi_2}{\pi_1}{(X_1, \varphi_1)}$ is onto and for $i=1$ or $2$ there exists $z \in X_i$ such that $(\eta_i)^{-1}\{z\}$ is a single element. Then $\eta_2 \times \eta_1 $ is a conjugacy.
\end{thm}

\begin{proof}
Lemma \ref{Lem:n-1} implies that $\eta_2 \times \eta_1$ is $n$-to-1 for some $n\ge 1$. The assumptions that $(\eta_i)^{-1}\{z\}$ is a single element and $\eta_i = P_i \circ (\eta_2 \times \eta_1)$ imply that $(\eta_2 \times \eta_1)^{-1}\{z\}$ is also a single element. Thus $n$ must be one and hence $\eta_2 \times \eta_1$ is a conjugacy.
\end{proof}

\begin{remark}
The previous theorem implies that the assumption that the relevant map is a conjugacy in the statements of Theorems \ref{mod5.4.1}, \ref{naturality} and \ref{main_thm} can be replaced by the assumption that the map is onto and the relevant set contains a single element in the case when the Smale spaces are irreducible.
\end{remark}

\begin{thm}[cf~{\cite[Theorem 5.4.1]{put}}]\label{mod5.4.1}
Suppose $(X,\phi)$ is a non-wandering Smale space with s/u-bijective pair $\pi = (Y,\psi,\pi_s,Z,\zeta,\pi_u)$ and $(X',\phi')$ is a non-wandering Smale space with s/u-bijective pair $\pi' = (Y',\psi',\pi'_s,Z',\zeta',\pi'_u)$.  Define $\eta = (\eta_X,\eta_Y,\eta_Z)$ to be a triple of factor maps
\begin{align*}
\eta_X: (X,\phi) & \to (X',\phi') \\
\eta_Y: (Y,\psi) & \to (Y',\psi') \\
\eta_Z: (Z,\zeta) & \to (Z',\zeta') \\
\end{align*}
such that the diagrams
\[
\begin{CD}
(Y,\psi) @>>\pi_s> (X,\varphi) \\
@VV \eta_Y V @VV \eta_X V \\
(Y',\psi') @>> \pi'_s > (X',\varphi') \\
\end{CD} \qquad \textrm{and} \qquad 
\begin{CD}
(Z,\zeta) @>>\pi_u> (X,\varphi) \\
@VV \eta_Z V @VV \eta_X V \\
(Z',\zeta') @>> \pi'_u > (X',\varphi') \\
\end{CD}
\]
are both commutative and the maps $\eta_Y \times \pi_s: (Y,\psi) \to  \fib{(Y', \psi')}{\pi'_s}{\eta_X}{(X , \varphi)}$ and $\eta_Z \times \pi_u: (Z,\zeta) \to  \fib{(Z', \zeta')}{\pi'_u}{\eta_X}{(X , \varphi)}$ are both conjugacies.
\begin{enumerate}
\item If $\eta_X$, $\eta_Y$, and $\eta_Z$ are s-bijective, then they induce chain maps between the complexes $C_{\QQ,\AA}^s(\pi)$ and $C_{\QQ,\AA}^s(\pi')$ and hence group homomorphisms
\[
\eta^s: H_N^s(\pi) \to H_N^s(\pi'),
\]
for every integer $N$.
\item If $\eta_X$, $\eta_Y$, and $\eta_Z$ are u-bijective, then they induce chain maps between the complexes $C_{\QQ,\AA}^s(\pi')$ and $C_{\QQ,\AA}^s(\pi)$ and hence group homomorphisms
\[
\eta^{s*}: H_N^s(\pi') \to H_N^s(\pi),
\]
for every integer $N$.

\end{enumerate}
These constructions are functorial in the sense that if $\eta_1$ and $\eta_2$ are both triples of s-bijective factor maps and the ranges of $\eta_1$ are the domains of $\eta_2$, then
\[
(\eta_2 \circ \eta_1)^s = \eta_2^s \circ \eta_1^s.
\]
An analogous statement holds for the composition of u-bijective factor maps.

\end{thm}
\begin{remark}
The statement of \cite[Theorem 5.4.1]{put} only requires the maps $\eta_Y \times \pi_s$ and $\eta_Z \times \pi_u$ to be surjective, not conjugacies as above. Again Example \ref{2to1example} shows that surjectivity alone is not enough. In the following section we show that it is always possible to find s/u-bijective pairs that satisfy our more stringent requirement (in fact, the s/u-bijective pair Putnam constructs in \cite[Theorem 5.4.2]{put} is an example); hence the conclusions regarding functoriality in \cite{put} remain valid. 
\end{remark}

\begin{proof}[Proof of Theorem \ref{mod5.4.1}]
The proof is as in \cite[Theorem 5.4.1]{put}.  However, since the proof of \cite[Theorem 5.4.1]{put} ultimately relies on \cite[Theorem 3.5.11]{put} (via \cite[Theorem 4.4.1]{put}), the maps $\eta_Y \times \pi_s$ and $\eta_Z \times \pi_u$ should be conjugacies.
\end{proof}

\section{Naturality of Putnam's homology theory}\label{naturality_homology}

In the previous section we observed that Theorem \ref{mod5.4.1} required a more stringent hypothesis than the one stated in \cite[Theorem 5.4.1]{put}. Theorem 5.5.1 in \cite{put} shows that $H^s_{*}(X,\phi)$ is independent of the s/u-bijective pair used to compute it, which allows the definition $H^s_{N}(X,\phi):=H^s_{N}(\pi)$ where $\pi$ is any bijective pair used to compute the homology. Provided the hypotheses of Theorem \ref{mod5.4.1} are satisfied, it implies that for an s-bijective map $\eta_X: (X,\phi) \to (X',\phi')$, there is an induced group homomorphism $\eta^s: H^s_{N}(X,\phi) \to H^s_{N}(X',\phi')$, where $\eta^s$ is defined in the statement of Theorem \ref{mod5.4.1}.  
In this section we ensure that the functoriality results remain valid given our more stringent requirements on s/u-bijective pairs. We also prove a naturality result with regard to Putnam's homology theory (see Theorem \ref{main_naturality} for the precise statement). In Theorem \ref{auto_naturality} we give a natural class of maps (i.e., automorphisms) that satisfy the hypotheses of the theorems in this section. This example is related to Putnam's Lefschetz formula (see Remark \ref{LF_rem}). 

Our first result shows that we can relax the hypotheses of Theorem \ref{mod5.4.1}. For example, using the notation of Theorem \ref{mod5.4.1}, if $\eta$ is s-bijective, then we need only assume that $\pi_u \times \eta_Z$ is a conjugacy; this allows slightly more flexibility in the choice of s/u-bijective pairs when defining the induced map on homology.

\begin{thm}\label{naturality}
Let $\pi = (Y,\psi,\pi_s,Z,\zeta,\pi_u)$ and $\pi' = (Y',\psi',\pi'_s,Z',\zeta',\pi'_u)$ be s/u-bijective pairs for the non-wandering Smale spaces $(X,\phi)$ and $(X',\phi')$ respectively.  Let $\eta = (\eta_X,\eta_Y,\eta_Z)$ be a triple of s-bijective factor maps, or a triple of u-bijective factor maps.
\begin{align*}
\eta_X: (X,\phi) & \to (X',\phi') \\
\eta_Y: (Y,\psi) & \to (Y',\psi') \\
\eta_Z: (Z,\zeta) & \to (Z',\zeta')
\end{align*}
such that the following diagram commutes:
\begin{equation}\label{com_diag_nat}
\begin{CD}
(Y,\psi) @>>\pi_s> (X,\varphi) @<<\pi_u< (Z,\zeta)\\
@VV \eta_Y V @VV \eta_X V @VV \eta_Z V\\
(Y',\psi') @>> \pi'_s > (X',\varphi') @<<\pi'_u< (Z',\zeta').\\
\end{CD}
\end{equation}
If $\eta$ is s-bijective, we assume $\pi_u \times \eta_Z: (Z, \zeta) \to \fib{(X, \phi)}{\eta_X}{\pi_u'}{(Z', \zeta')}$ is a conjugacy.
Then, for each $L,M \geq 0$ 
\[
\begin{CD}
\Sigma_{L,M}(\pi) @>>\eta_{L,M}> \Sigma_{L,M}(\pi') \\
@VV \delta_{,m} V @VV \delta_{,m} V \\
\Sigma_{L,M-1}(\pi) @>> \eta_{L,M-1} > \Sigma_{L,M-1}(\pi') \\
\end{CD}
\]
is a commutative diagram and $\eta_{L,M} \times \delta_{,m}: \Sigma_{L,M}(\pi) \to \fib{\Sigma_{L,M}(\pi')}{\delta'_{,m}}{\eta_{L,M-1}}{\Sigma_{L,M-1}(\pi)}$ is a conjugacy.
In particular, $\eta$ induces a chain map on the double complexes used to define $H^s(\pi)$ and $H^s(\pi')$. 

Similarly, if $\eta$ is u-bijective we assume in \eqref{com_diag_nat} that $\pi_s \times \eta_Y: (Y, \psi) \to \fib{(X, \phi)}{\eta_X}{\pi_s'}{(Y', \psi')}$ is a conjugacy.
Then, for each $L,M \geq 0$
\[
\begin{CD}
\Sigma_{L,M}(\pi) @>>\delta_{l,} > \Sigma_{L-1,M}(\pi) \\
@VV \eta_{L,M} V @VV \eta_{L-1,M} V \\
\Sigma_{L,M}(\pi') @>> \delta_{l,} > \Sigma_{L-1,M}(\pi') \\
\end{CD}
\]
is a commutative diagram and  $\delta_{l,} \times \eta_{L,M}: \Sigma_{L,M}(\pi) \to \fib{\Sigma_{L-1,M}(\pi)}{\eta_{L-1,M}}{\delta_{l,}}{\Sigma_{L,M}(\pi')}$ is a conjugacy.
In particular, $\eta$ induces a chain map on the double complexes used to define $H^s(\pi)$ and $H^s(\pi')$.
Moreover, these constructions are functorial in the same sense as in the statement of Theorem \ref{mod5.4.1}.
\end{thm}
\begin{proof}
Assume $\eta$ is s-bijective. The u-bijective case is similar and we omit the details.
We consider the case of $L=M=1$, the proof of the general case is analogous. We must show that $\eta_{L,M} \times \delta_{,m}$ is a conjugacy.  It is clear that it intertwines the dynamics, so it suffices to show that it is a bijection.

We first show that $\eta_{L,M} \times \delta_{,m}$ is surjective.  Let $a = (y'_0,y'_1,z'_0,z'_1) \in \Sigma_{1,1}(\pi')$ and let $b = (y_0,y_1,z_0) \in \Sigma_{1,0}(\pi)$. This implies that 
\begin{itemize}
\item $\pi'_s(y'_0)=\pi'_s(y'_1)=\pi'_u(z'_0)=\pi'_u(z'_1)$
\item $\pi_s(y_0) = \pi_s(y_1) = \pi_u(z_0)$
\item $(y'_0,y'_1,z_0) = (\eta_Y(y_0),\eta_Y(y_1),\eta_Z(z_0))$
\end{itemize}
we will show that $(a,b)$ is the image of some $c \in \Sigma_{1,1}(\pi)$ under the map $\eta_{1,1} \times \delta_{,m}$.  Let $c = (y_0,y_1,z_0,z)$ where $y_0,y_1,z_0$ are as above, and $z$ is given as follows.
Since
\[
\eta_X(\pi_s(y_1)) = \pi'_s(\eta_Y(y_1)) = \pi'_s(y'_1) = \pi'_u(z'_1)
\]
we have that $(z'_1,\pi_s(y_1)) \in \fib{(Z',\zeta')}{\pi'_u}{\eta_X}{(X,\phi)}$. Now, $\eta_Z \times \pi_u : (Z,\zeta) \to (Z',\zeta')_{\pi'_u}\times_{\eta_X} (X,\phi)$ is onto, so there exists $z \in Z$ such that $(\eta_Z(z),\pi_u(z)) = (z'_1,\pi_s(y_1))$.
We must now show that $c = (y_0,y_1,z_0,z) \in \Sigma_{L,M}(\pi_s,\pi_u)$. We have that $\pi_s(y_0) = \pi_s(y_1)=\pi_u(z_0)$ holds since $(y_0,y_1,z_0) \in \Sigma_{1,0}(\pi)$, and from the definition of $z$, we have $\pi_u(z) = \pi_s(y_1)$. So we compute 
\[
\eta_{1,1} \times \delta_{,m}(y_0,y_1,z_0,z)  = \big((\eta_Y(y_0),\eta_Y(y_1),\eta_Z(z_0),\eta_Z(z)),(y_0,y_1,z_0)\big) = \big((y'_0,y'_1,z'_0,z'_1),(y_0,y_1,z_0)\big),
\]
and $\eta_{L,M} \times \delta_{,m}$ is surjective.

We now show that $\eta_{L,M} \times \delta_{,m}$ is injective.  Let $a = (y_0,y_1,z_0,z_1)$ and $\tilde{a} = (\tilde{y_0},\tilde{y_1},\tilde{z_0},\tilde{z_1}) \in \Sigma_{1,1}(\pi)$, with $(\eta_{L,M} \times \delta_{,m})(a) = (\eta_{L,M} \times \delta_{,m})(\tilde{a})$. Then, since $\delta_{,1}(\tilde{a}) = \delta_{,1}(a)$ we have $(y_0,y_1,z_0) = (\tilde{y_0},\tilde{y_1},\tilde{z_0})$ so we need only show that $z_1 = \tilde{z_1}$.  Now $\eta_{1,1}(\tilde{a}) = \eta_{1,1}(a)$ implies $\eta_Z(z_1) = \eta_Z(\tilde{z_1})$, and because $y_1 = \tilde{y_1}$ and $\pi_s$, we have $\pi_s(y_1) = \pi_s(\tilde{y_1})$. Thus
\[
\pi_u(z_1) = \pi_s(y_1) = \pi_s(\tilde{y_1}) = \pi_u(\tilde{z_1}),
\]
and we have shown that both $\eta_Z(z_1) = \eta_Z(\tilde{z_1})$ and $\pi_u(z_1) = \pi_u(\tilde{z_1})$. Since $\eta_Z \times \pi_u$ is injective by hypothesis, we must have that $z_1 \sim_s \tilde{z_1}$. So we have proven that $\tilde{a} = a$, and $\eta_{L,M} \times \delta_{,m}$ is injective.

Applying Theorem \ref{mod3.5.11} shows that 
\[
\delta_{,m}^{s\ast} \circ \eta_{L,M-1}^s = \eta_{L,M}^s \circ \delta_{,m}^{s\ast}: D^s(\Sigma_{L,M-1}(\pi)) \to D^s(\Sigma_{L,M}(\pi')).
\]
In other words, $\eta$ induces a chain map on the double complexes used to define $H^s(\pi)$ and $H^s(\pi')$.

That the construction is functorial follows as in the proof of \cite[Theorem 5.4.1]{put}.
\end{proof}

In the case that $\eta$ is s-bijective (u-bijective), we will define $\eta^s: H^s(X,\phi) \to H^s(X',\phi')$ (respectively $\eta^{s*}: H^s(X',\phi') \to H^s(X,\phi)$) using only s/u-bijective pairs which satisfy the hypotheses of the previous theorem (also see the statement of Theorem \ref{mod5.4.1}). Our next two results show:
\begin{enumerate}
\item the existence of s/u-bijective pairs satisfying the previous theorem, and
\item the definition of $\eta^s$ (and $\eta^{s*}$) is ``natural" with respect to the construction of $H^s(X,\phi)$. 
\end{enumerate}
The next theorem is (more or less) \cite[Theorem 5.4.2]{put}. However, since it answers a rather natural and important question, we include it and its proof.
\begin{thm}
Let $\eta_X:(X,\phi) \to (X',\phi')$ be a factor map between non-wandering Smale spaces. 
 If $\eta_X$ is s-bijective (u-bijective), then there exist s/u-bijective pairs for $(X,\phi)$ and $(X',\phi')$ which satisfy the appropriate hypotheses of Theorem \ref{naturality}.
\end{thm}
\begin{proof}
We prove only the case that $\eta_X$ is s-bijective. The u-bijective case is analogous.
Let $(Z',\zeta')$ be any choice of ``u-part'' of an s/u-bijective pair for $(X',\phi')$, and $(Y,\psi)$ be any choice of ``s-part'' of an s/u-bijective pair for $(X,\phi)$ (the existence of these is guaranteed by \cite[Theorem 2.6.3]{put}). We then have the following commutative diagram of s/u-bijective pairs for $(X,\phi)$ and $(X',\phi')$
\[
\begin{CD}
(Y,\psi) @>>\pi_s> (X,\varphi) @<<P_2< \fib{(Z',\zeta')}{\pi'_u}{\eta_X}{(X,\phi)}\\
@VV id V @VV \eta_X V @VV P_1 V\\
(Y,\psi) @>> \eta_X \circ \pi_s > (X',\varphi') @<< \pi'_u< (Z',\zeta')\\
\end{CD}
\]
That these form s/u-bijective pairs follows from the proof of \cite[Theorem 5.4.2]{put}.  The extra hypothesis that the upper right hand corner be conjugate to the fibre product is immediate.
\end{proof}

\begin{thm}\label{main_naturality}
Let $\eta_X:(X,\phi) \to (X',\phi')$ be a factor map between non-wandering Smale spaces along with the following commuting diagrams 
\[
\begin{CD}
(Y,\psi) @>>\pi_s> (X,\varphi) @<<\pi_u< (Z,\zeta)\\
@VV \eta_Y V @VV \eta_X V @VV \eta_Z V\\
(Y',\psi') @>> \pi'_s > (X',\varphi') @<<\pi'_u< (Z',\zeta')\\
\end{CD}
\ \ \textrm{and} \ \ 
\begin{CD}
(\tilde{Y},\tilde{\psi}) @>>\tilde{\pi}_s> (X,\varphi) @<<\tilde{\pi}_u< (\tilde{Z},\tilde{\zeta})\\
@VV \tilde{\eta}_Y V @VV \eta_X V @VV \tilde{\eta}_Z V\\
(\tilde{Y}',\tilde{\psi}') @>> \tilde{\pi}'_s > (X',\varphi') @<<\tilde{\pi}'_u< (\tilde{Z}',\tilde{\zeta}')\\
\end{CD}
\]
\begin{enumerate}
\item If $\eta_X$ is s-bijective, and the above diagrams represent two different sets of data which both satisfy the hypotheses of Theorem \ref{naturality}, then
\[
\Theta' \circ \eta^s = \tilde{\eta}^s \circ \Theta
\]
where $\eta^s$ is defined via the the first diagram (via $\eta = (\eta_X,\eta_Y,\eta_Z$)), $\tilde{\eta}^s$ is defined via the second diagram (via $\tilde{\eta} = (\eta_X,\tilde{\eta_Y},\tilde{\eta_Z}$)), and $\Theta': H^s(\pi') \to H^s(\tilde{\pi}')$ and $\Theta: H^s(\pi) \to H^s(\tilde{\pi})$ are the isomorphisms described in the proof of \cite[Theorem 5.5.1]{put}.
\item If $\eta_X$ is u-bijective, and the above diagrams represent two different sets of data which both satisfy the hypotheses of Theorem \ref{naturality}, then
\[
\Theta \circ \eta^{s*} = \tilde{\eta}^{s*} \circ \Theta'
\]
where $\eta^{s*}$ is defined via the the first diagram (via $\eta = (\eta_X,\eta_Y,\eta_Z$)), $\tilde{\eta}^{s*}$ is defined via the second diagram (via $\tilde{\eta} = (\eta_X,\tilde{\eta_Y},\tilde{\eta_Z}$)), and $\Theta': H^s(\pi') \to H^s(\tilde{\pi}')$ and $\Theta: H^s(\pi) \to H^s(\tilde{\pi})$ are the isomorphisms described in the proof of \cite[Theorem 5.5.1]{put}.
\end{enumerate}
\end{thm}
\begin{proof}
We prove only the case that $\eta_X$ is s-bijective, the u-bijective case is similar. 

We first consider the case that $Y = \tilde{Y}$, $\psi = \tilde{\psi}$, $\pi_s = \tilde{\pi}_s$, and $Y' = \tilde{Y}'$, $\psi' = \tilde{\psi}'$, $\pi'_s = \tilde{\pi}'_s$. Let  We denote $(\eta_X,\eta_Y,\eta_Z\times\tilde{\eta}_Z)$ by $\tilde{\tilde{\eta}}$,
and show that the diagram
\begin{equation}\label{Diag:CD_proofNat}
\begin{CD}
(Y,\psi) @>>\pi_s> (X,\varphi) @<<\tilde{\pi}_u \circ P_2< \fib{(Z,\zeta)}{\pi_u}{\tilde{\pi}_u}{(\tilde{Z},\tilde{\zeta})}\\
@VV \eta_Y V @VV \eta_X V @VV \eta_Z \times \tilde{\eta}_Z V\\
(Y',\psi') @>> \pi'_s > (X',\varphi') @<<\tilde{\pi}'_u \circ P'_2< \fib{(Z',\zeta')}{\pi'_u}{\tilde{\pi}'_u}{(\tilde{Z}',\tilde{\zeta}')}\\
\end{CD}
\end{equation}
satisfies the hypotheses of Theorem \ref{naturality} and hence can be used to define a map ($\tilde{\tilde{\eta}}^s$) on homology.  We need only check that $(\tilde{\pi}_u \circ P_2) \times (\eta_Z \times \tilde{\eta}_Z)$ is a conjugacy onto the fibre product.

To show surjectivity, suppose $(x,z',\tilde{z}') \in \fib{(X,\varphi)}{\eta_X}{\tilde{\pi}'_u \circ P'_2}{\big(\fib{(Z',\zeta')}{\pi'_u}{\tilde{\pi}'_u}{(\tilde{Z}',\tilde{\zeta}')}\big)}$, then $(x,\tilde{z}') \in \fib{(X,\phi)}{\eta_X}{\tilde{\pi}'_u}{(\tilde{Z}',\tilde{\zeta}')}$ and $(x,z') \in \fib{(X,\phi)}{\eta_X}{\pi'_u}{(Z',\zeta')}$. Now recall that the right-hand squares in each of the diagrams in the statement of this theorem satisfy the conjugacy condition, hence there exist $z \in Z$ and $\tilde{z} \in \tilde{Z}$ such that $(\pi_u \times \eta_Z)(z) = (x,z')$ and $(\tilde{\pi}_u \times \tilde{\eta}_Z)(\tilde{z}) = (x,\tilde{z}')$.  We now have that 
\[
\big((\tilde{\pi}'_u \circ P'_2) \times (\eta_Z \times \tilde{\eta}_Z)  \big)(z,\tilde{z}) = (x,z',\tilde{z}') 
\]
and hence the map is surjective.

Now we show that $(\tilde{\pi}_u \circ P_2) \times (\eta_Z \times \tilde{\eta}_Z)$ is injective. Suppose $(z_1,\tilde{z}_1),(z_2,\tilde{z}_2) \in \fib{(Z,\zeta)}{\pi_u}{\tilde{\pi}_u}{(\tilde{Z},\tilde{\zeta})}$ such that $(\tilde{\pi}_u \circ P_2)(z_1,\tilde{z}_1) = (\tilde{\pi}_u \circ P_2)(z_2,\tilde{z}_2)$ and $(\eta_Z \times \tilde{\eta}_Z)(z_1,\tilde{z}_1) = (\eta_Z \times \tilde{\eta}_Z)(z_2,\tilde{z}_2)$.  Then
\[
\pi_u(z_1) = \tilde{\pi}_u(\tilde{z}_1) = (\tilde{\pi}_u \circ P_2)(z_1,\tilde{z}_1) = (\tilde{\pi}_u \circ P_2)(z_2,\tilde{z}_2) = \tilde{\pi}_u(\tilde{z}_2) = \pi_u(z_2).
\]
Also, $\eta_Z(z_1) = \eta_Z(z_2)$ and, $\tilde{\eta}_Z(\tilde{z}_1) = \tilde{\eta}_Z(\tilde{z}_2)$. Now, recalling that the two diagrams in the statement of the theorem each satisfy the conjugacy condition, we see that $z_1 = z_2$ and $\tilde{z}_1 = \tilde{z}_2$, so $(z_1,\tilde{z}_1) = (z_2,\tilde{z}_2)$ and the map is injective.

We can therefore define a map $\tilde{\tilde{\eta}}^s$ on homology using diagram \ref{Diag:CD_proofNat} (via $\tilde{\tilde{\eta}} = (\eta_X,\eta_Y,\eta_Z\times\tilde{\eta}_Z)$). Notice also that if we apply $P_1$ and $P'_1$ to the fibred products on the right hand side of this diagram (and the identity map everywhere else) we get the following diagram
\[
\begin{CD}
(Y,\psi) @>>\pi_s> (X,\varphi) @<<\pi_u< (Z,\zeta)\\
@VV \eta_Y V @VV \eta_X V @VV \eta_Z V\\
(Y',\psi') @>> \pi'_s > (X',\varphi') @<<\pi'_u< (Z',\zeta')\\
\end{CD}
\]
hence $P_1$ induces an isomorphism 
\[
\Theta_{P_1}: H^s(Y,\pi_s, \fib{(Z,\zeta)}{\pi_u}{\tilde{\pi}_u}{(\tilde{Z},\tilde{\zeta})},\tilde{\pi}_u \circ P_2) \to H^s(Y, \pi_s, Z, \pi_u)
\]
and  $P'_1$ induces an isomorphism 
\[
\Theta_{P'_1}: H^s(Y',\pi'_s, \fib{(Z',\zeta')}{\pi'_u}{\tilde{\pi}'_u}{(\tilde{Z}',\tilde{\zeta}')},\tilde{\pi}'_u \circ P'_2) \to H^s(Y', \pi'_s, Z', \pi'_u),
\]
see \cite[Theorem 5.5.1]{put} for details.  It suffices to show that $\Theta_{P'_1} \circ \tilde{\tilde{\eta}}^s = \eta^s \circ \Theta_{P_1}$. In addition we must prove a similar result for $P_2$ and $P'_2$ and then use transitivity, however, the proof is analogous, and hence omitted. 

In order to prove $\Theta_{P'_1} \circ \tilde{\tilde{\eta}}^s = \eta^s \circ \Theta_{P_1}$ we begin by showing that $P_1$ induces a map on the double complexes (that the induced map is an isomorphism follows from \cite[Theorem 5.5.1]{put}).  Since $P_1$ is u-bijective, it suffices to show that 
\[
\begin{CD}
\Sigma_{L,M}(\pi_s,\tilde{\pi}_u \circ P_2) @>>\tilde{\delta}_{l,}> \Sigma_{L-1,M}(\pi_s,\tilde{\pi}_u\circ P_2)\\
@VV P_1 V @VV P_1 V \\
\Sigma_{L,M}(\pi_s,\pi_u) @>>\delta_{l,}> \Sigma_{L-1,M}(\pi_s,\pi_u)\\
\end{CD}
\]
satisfies the hypotheses of Theorem \ref{naturality}, i.e. that 
\[
\tilde{\delta}_{l,} \times P_1: \Sigma_{L,M}(\pi_s,\tilde{\pi}_u\circ P_2) \to \fib{\Sigma_{L-1,M}(\pi_s,\tilde{\pi}_u\circ P_2)}{P_1}{\delta_{l,}}{\Sigma_{L,M}(\pi_s,\pi_u)}
\]
is a conjugacy.

Consider the map $\Phi: \fib{\Sigma_{L-1,M}(\pi_s,\tilde{\pi}_u\circ P_2)}{P_1}{\delta_{l,}}{\Sigma_{L,M}(\pi_s,\pi_u)} \to \Sigma_{L,M}(\pi_s,\tilde{\pi}_u\circ P_2)$ given by 
\[
\Phi(\delta_{l}(y_i)_{i=0}^L,(\tilde{z}_j)_{j=0}^M,(y_i)_{i=0}^L,(\tilde{\pi}_u(\tilde{z}_j))_{j=0}^M) = ((y_i)_{i=0}^L,(\tilde{z}_j)_{j=0}^M),
\]
where $\delta_l$ deletes the $l$th coordinate as in \eqref{delete_n_Cech}.
It is straightforward to check that $\Phi$ is the inverse of $\tilde{\delta}_{l,} \times P_1$.  Therefore, the above diagram satisfies the additional hypotheses of Theorem \ref{naturality} and induces a map on the double complexes, and hence this map is an isomorphism (as in the proof of \cite[Theorem 5.5.1]{put}).

Now consider the diagram
\[
\begin{CD}
\fib{(Z,\zeta)}{\pi_u}{\tilde{\pi}_u}{(\tilde{Z},\tilde{\zeta})} @>>\eta_Z \times \tilde{\eta}_Z> \fib{(Z',\zeta')}{\pi'_u}{\tilde{\pi}'_u}{(\tilde{Z}',\tilde{\zeta}')}\\
@VV P_1 V @VV P_1' V\\
(Z,\zeta) @>>\eta_Z> (Z',\zeta')\\
\end{CD}
\]
we must show that $(\eta_Z \times \tilde{\eta}_Z) \times P_1$ is a conjugacy onto the fibre product.

For surjectivity, suppose $(z',\tilde{z}',z) \in \fib{(\fib{(Z',\zeta')}{\pi'_u}{\tilde{\pi}'_u}{(\tilde{Z}',\tilde{\zeta}')})}{P'_1}{\eta_Z}{(Z,\zeta)}$, then the fiber product conditions imply $z' = \eta_Z(z)$ and $ \pi'_u(z') =\tilde{\pi}'_u(\tilde{z}')$. Hence an arbitrary element looks like $(\eta_Z(z),\tilde{z}',z)$ where $\tilde{\pi}'_u(\tilde{z}') = \pi'_u(\eta_Z(z)) = \eta_X(\pi_u(z))$. So $(\pi_u(z),\tilde{z}') \in \fib{(X,\phi)}{\eta_X}{\pi'_u}{(\tilde{Z}',\tilde{\zeta}')}$ and since 
\[
\tilde{\pi}_u \times \tilde{\eta}_Z: (\tilde{Z},\tilde{\zeta}) \to \fib{(X,\phi)}{\eta_X}{\pi'_u}{(\tilde{Z}',\tilde{\zeta}')}
\]
is onto by assumption, there exists $\tilde{z} \in (\tilde{Z},\tilde{\zeta})$ such that $\tilde{\pi}_u \times \tilde{\eta}_Z(\tilde{z})=(\pi_u(z),\tilde{z}')$. By construction, $\tilde{\pi}_u(\tilde{z}) = \pi_u(z)$, so $(z,\tilde{z}) \in (Z,\zeta)_{\pi_u}\times_{\tilde{\pi}_u}(\tilde{Z},\tilde{\zeta})$.  Moreover
$$
\big((\eta_Z \times \tilde{\eta}_Z )\times P_1\big)(z,\tilde{z}) = (\eta_Z(z), \tilde{\eta}_Z (z), z) = (\eta_Z(z) , \tilde{z}',z).
$$
Hence $(\eta_Z \times \tilde{\eta}_Z )\times P_1$ is surjective.

For injectivity, suppose $(z_1, \tilde{z}_1), (z_2, \tilde{z}_2) \in \fib{(Z,\zeta)}{\pi_u}{\tilde{\pi}_u}{(\tilde{Z},\tilde{\zeta})} $ such that $(\eta_Z \times \tilde{\eta}_Z)(z_1, \tilde{z}_1) = (\eta_Z \times \tilde{\eta}_Z)(z_2, \tilde{z}_2)$ and $P_1(z_1, \tilde{z}_1) = P_1(z_2, \tilde{z}_2)$. In other words $\eta_Z(z_1) = \eta_Z(z_2)$, $\tilde{\eta}_Z(\tilde{z}_1) = \tilde{\eta}_Z(\tilde{z}_2)$, and $z_1 = z_2$. To complete this portion of the proof we must show that $(z_1,\tilde{z}_1) = (z_2,\tilde{z}_2)$.  Since we already have $z_1 = z_2$ we need only show $\tilde{z}_1 = \tilde{z}_2$.  Towards this end, recall that the diagram
\[
\begin{CD}
(\tilde{Z},\tilde{\zeta}) @>> \tilde{\eta}_Z> (\tilde{Z}',\tilde{\zeta}')\\
@VV \tilde{\pi}_u V @VV \tilde{\pi}'_u V\\
(X,\phi) @>>\eta_X> (X',\phi')\\
\end{CD}
\]
satisfies the conjugacy condition.  Moreover, $\tilde{\eta}_Z(\tilde{z}_1) = \tilde{\eta}_Z(\tilde{z}_2)$ and 
\[
\tilde{\pi}_u(\tilde{z}_1) = \pi_u(z_1) = \pi_u(z_2) = \tilde{\pi}_u(\tilde{z}_2).
\]
Thus $\tilde{z}_1 = \tilde{z}_2$ and the map is injective.

We have shown that
\[
\begin{CD}
\fib{(Z,\zeta)}{\pi_u}{\tilde{\pi}_u}{(\tilde{Z},\tilde{\zeta})} @>>\eta_Z \times \tilde{\eta}_Z> \fib{(Z',\zeta')}{\pi'_u}{\tilde{\pi}'_u}{(\tilde{Z}',\tilde{\zeta}')}\\
@VV P_1 V @VV P_1' V\\
(Z,\zeta) @>>\eta_Z> (Z',\zeta')\\
\end{CD}
\]
satisfies the conjugacy condition, and hence, for $L,M \geq 0$, the diagram
\[
\begin{CD}
\Sigma_{L,M}(\pi_s,\tilde{\pi}_u \circ P_2) @>>(\eta_y,\eta_Z \times \tilde{\eta}_Z)> \Sigma_{L,M}(\pi'_s,\tilde{\pi}'_u\circ P'_2)\\
@VV P_1 V @VV P'_1 V \\
\Sigma_{L,M}(\pi_s,\pi_u) @>>(\eta_Y,\eta_Z)> \Sigma_{L,M}(\pi'_s,\pi'_u)\\
\end{CD}
\]
also satisfies the conjugacy condition. Which in turn implies that  $\Theta_{P'_1} \circ \tilde{\tilde{\eta}}^s = \eta^s \circ \Theta_{P_1}$, the required result.

We have therefore proved the result in the case that $Y = \tilde{Y}$, $\psi = \tilde{\psi}$, $\pi_s = \tilde{\pi}_s$, and $Y' = \tilde{Y}'$, $\psi' = \tilde{\psi}'$, $\pi'_s = \tilde{\pi}'_s$.  To prove the general result it suffices to now prove the result in the case that $Z = \tilde{Z}$, $\zeta = \tilde{\zeta}$, $\pi_u = \tilde{\pi}_u$, and $Z' = \tilde{Z}'$, $\zeta' = \tilde{\zeta}'$, $\pi'_u = \tilde{\pi}'_u$.  This result is similar, however the details are less complicated since in this case all of the maps are s-bijective and we can appeal directly to the functoriality of s-bijective maps without needing to use Theorem \ref{naturality}. We omit the details of this part of the proof.
\end{proof}

\begin{thm}\label{auto_naturality}
Suppose $\pi = (Y,\psi,\pi_s,Z,\zeta,\pi_u)$ is an s/u-bijective pair for a non-wandering Smale space $(X,\phi)$. Let 
\begin{align*}
\alpha_X: (X,\phi) & \to (X,\phi) \\
\alpha_Y: (Y,\psi) & \to (Y,\psi) \\
\alpha_Z: (Z,\zeta) & \to (Z,\zeta)
\end{align*}
be automorphisms such that the following diagram commutes:
\begin{equation}\label{com_diag_nat1}
\begin{CD}
(Y,\psi) @>>\pi_s> (X,\varphi) @<<\pi_u< (Z,\zeta)\\
@VV \alpha_Y V @VV \alpha_X V @VV \alpha_Z V\\
(Y,\psi) @>> \pi_s > (X,\varphi) @<<\pi_u< (Z,\zeta).\\
\end{CD}
\end{equation}
Then 
\begin{align}
\label{eq1_auto}
\pi_u \times \alpha_Z&: (Z, \zeta) \to \fib{(X, \phi)}{\alpha_X}{\pi_u}{(Z, \zeta)} \quad \text{ and } \\
\pi_s \times \alpha_Y&: (Y, \psi) \to \fib{(X, \phi)}{\alpha_X}{\pi_s}{(Y, \psi)} \label{eq2_auto}
\end{align}
are conjugacies.
\end{thm}
\begin{proof}
We prove \eqref{eq1_auto}; the proof of \eqref{eq2_auto} is similar. The commutative diagram \eqref{com_diag_nat1} reduces the proof to showing that $\pi_u \times \alpha_Z$ is bijective. Let $\Phi=\alpha_Z^{-1} \circ P_2:\fib{(X, \phi)}{\alpha_X}{\pi_u}{(Z, \zeta)} \to (Z, \zeta)$ (i.e., $\Phi(x,z)=\alpha_Z^{-1}(z)$).

We show that $\Phi$ is the inverse to $\pi_u \times \alpha_Z$. Indeed, on the one hand, we have
\begin{align*}
\big((\pi_u \times \alpha_Z) \circ \Phi\big)(x,z)&=(\pi_u \times \alpha_Z)(\alpha_Z^{-1}(z)) \\
&= (\pi_u(\alpha_Z^{-1}(z)),z) \\
&= (\alpha_X^{-1}(\pi_u(z)),z) \\
&= (\alpha_X^{-1}(\alpha_X(x)),z) = (x,z),
\end{align*}
where the second to last equality follows because $(x,z) \in \fib{(X, \phi)}{\alpha_X}{\pi_u}{(Z, \zeta)}$. On the other hand,
\[
\big(\Phi \circ (\pi_u \times \alpha_Z)\big)(z)=\Phi(\pi_u(z),\alpha_Z(z))=\alpha_Z^{-1}(\alpha_Z(z))=z.
\]
Hence $\Phi$ is the inverse of $\pi_u \times \alpha_Z$, completing the proof.
\end{proof}

\begin{remark}\label{LF_rem}
Assuming the hypotheses and notation of Theorem \ref{auto_naturality}, we have that the triple $(\alpha_X,\alpha_Y,\alpha_Z)$ and the s/u-bijective pair $\pi$ satisfy the hypotheses of Theorems \ref{mod5.4.1} and \ref{main_naturality}. A particularly relevant example is the case when $\alpha_X=\phi^n$, $\alpha_Y=\psi^n$, and $\alpha_Z=\zeta^n$ for some $n \in \Z$. These triples were used in \cite[Chapter 6]{put} to prove a Lefschetz formula for Smale spaces. Hence Putnam's arguments regarding the Lefschetz formula do not require change, even in light of the issues in \cite[Theorem 4.4.1]{put}. For the Lefschetz formula, Putnam considers both $H^s(X,\phi)$ and $H^u(X,\phi)$; the relevant results for $H^u(X,\phi)$ are discussed in Section \ref{unstable_results}.
\end{remark}

\section{A generalization of Putnam's pullback diagram}\label{main_theorem_section}

In this section we prove a result which generalizes Theorem \ref{mod3.5.11} (i.e., \cite[Theorem 3.5.11]{put}) from shifts of finite type to non-wandering Smale spaces. Putnam's Theorem 3.5.11 is ubiquitous throughout \cite{put} because it shows that Putnam's homology is functorial at the level of the building blocks of the theory. In order to see why this is the case requires a careful inspection of the proof of \cite[Theorem 5.4.1]{put} which heavily relies on \cite[Theorem 4.4.1]{put} whose proof is an almost direct consequence of \cite[Theorem 3.5.11]{put}.

\begin{thm}\label{main_thm}
Suppose 
\[
\begin{CD}
(X,\varphi) @>>\eta_1> (X_1,\varphi_1) \\
@VV \eta_2 V @VV \pi_1 V \\
(X_2,\varphi_2) @>> \pi_2 > (X_0,\varphi_0) \\
\end{CD}
\]
is a commutative diagram of non-wandering Smale spaces in which $\eta_1$ and $\pi_2$ are s-bijective factor maps, and $\eta_2$ and $\pi_1$ are u-bijective factor maps. Moreover, suppose 
$\eta_2 \times \eta_1 : (X, \varphi) \to \fib{(X_2, \varphi_2)}{\pi_2}{\pi_1}{(X_1, \varphi_1)}$ is a conjugacy.
Then
\[
\eta_1^{s} \circ \eta_2^{s*} = \pi_1^{s*} \circ \pi_2^{s}: H_N^s(X_2,\varphi_2) \to H_N^s(X_1,\varphi_1).
\]
\end{thm}

We break the proof of the theorem into a number of lemmas. The general strategy is to take an s/u-bijective pair for $(X_0, \varphi_0)$ and construct pairs for $(X_1,\varphi_1)$, $(X_2, \varphi_2)$, and $(X,\varphi)$ whose associated maps fit into a (rather ominous looking, but natural) commutative diagram (see \eqref{3d_diagram} on page \pageref{3d_diagram}).

In the general case, this construction is rather long (although not overly difficult). However, there is a particularly nice special case - namely, when $(X_2,\varphi_2,\pi_2,X_1,\varphi_1,\pi_1)$ is an s/u-bijective pair for $(X_0,\varphi_0)$. In this case, one has that the double complex for $(X_0,\varphi_0)$ contains the complex for $(X_1,\varphi_1)$ along its bottom row, the complex for $(X_2,\varphi_2)$ as its left column, and $(X,\varphi)$ as its lower left entry (see Definition \ref{HomDefn} or \cite[Definition 5.1.11]{put} for the double complex in question). In this case, the proof of Theorem \ref{main_thm} simplifies significantly.

Returning to the general situation, for the remainder of this section, fix the following diagram of non-wandering Smale spaces:
\begin{equation}\label{main_diagram}
\begin{CD}
(X,\varphi) @>>\eta_1> (X_1,\varphi_1) \\
@VV \eta_2 V @VV \pi_1 V \\
(X_2,\varphi_2) @>> \pi_2 > (X_0,\varphi_0). \\
\end{CD}
\end{equation}

\begin{lemma}\label{Lem1main}
Suppose we have the commutative diagram \eqref{main_diagram} satisfying the hypotheses of Theorem \ref{main_thm}. Let $(Y_0, \psi_0)$ be a non-wandering Smale space with totally disconnected unstable sets such that $\rho^{YX}_0 : Y_0 \to X_0$ is s-bijective.
Define 
\[
(Y_1, \psi_1) = \fib{(Y_0, \psi_0)}{\rho^{YX}_0}{\pi_1}{(X_1 , \varphi_1)}.
\]
Then 
\begin{enumerate}
\item $Y_1$ has totally disconnected unstable sets
\item $\rho^{YX}_1 = P_2 : Y_1 \to X_1$ is s-bijective
\item $\pi^{Y}_1 = P_1 : Y_1 \to Y_0$ is u-bijective
\end{enumerate}
\end{lemma}
\begin{proof}
Parts (2) and (3) follow immediately from \cite[Theorem 2.5.13]{put}.  Part (1) then follows from (3) and the fact that $Y_0$ has totally disconnected unstable sets.
\end{proof}

\begin{lemma}\label{Lem2main}
Suppose we have the commutative diagram \eqref{main_diagram} satisfying the hypotheses of Theorem \ref{main_thm} and  $(\tilde{Y}_2, \tilde{\psi}_2)$ is a non-wandering Smale space with totally disconnected unstable sets such that $\tilde{\rho}^{YX}: \tilde{Y}_2 \to X_2$ is s-bijective. Define
\[
(Y_2, \psi_2) = \fib{(Y_0, \psi_0)}{\rho^{YX}_0}{\pi_2 \circ \tilde{\rho}^{YX}}{(\tilde{Y}_2, \tilde{\psi}_2)},
\] 
then
\begin{enumerate}
\item $Y_2$ has totally disconnected unstable sets
\item $\rho^{YX}_2 = \tilde{\rho}^{YX} \circ P_2 : Y_2 \to X_2$ is s-bijective
\item $\pi^{Y}_2 = P_1 : Y_2 \to Y_0$ is s-bijective
\end{enumerate}
\end{lemma}
\begin{proof}
Part (2) and (3) again follow immediately from \cite[Theorem 2.5.13]{put}.  For (1) notice that for any point $(y_0, \tilde{y}_2) \in Y_2$ and any $\epsilon > 0$,  $Y_2^u((y_0, \tilde{y}_2), \epsilon) \subset Y_0^u(y_0, \epsilon) \times \tilde{Y}_2^u(\tilde{y}_2, \epsilon)$, which is totally disconnected as each term in the product is totally disconnected.
\end{proof}

\begin{lemma}\label{Lem3main}
Suppose we have the commutative diagram \eqref{main_diagram} satisfying the hypotheses of Theorem \ref{main_thm}, $(Y_1,\psi_1)$ is the Smale space described in Lemma \ref{Lem1main}, and $(Y_2,\psi_2)$ is the Smale space described in Lemma \ref{Lem2main}. Using fibre products, define a Smale space
\begin{equation}\label{DoubleFibre}
(Y,\psi):=\{(y_2, x, y_1) \in \fib{(Y_2,\psi_2)}{\rho^{YX}_2}{\eta_2}{(X, \varphi)} \fib{}{\eta_1}{\rho^{YX}_1}{(Y_1, \psi_1)} \mid \pi^{Y}_2(y_2) = \pi^{Y}_1(y_1)\},
\end{equation}
then
\begin{enumerate}
\item $Y$ has totally disconnected unstable sets
\item $\eta^Y_2 = P_1: Y \to Y_2$ is u-bijective
\item $\rho^{YX} = P_2: Y \to X$ is s-bijective
\item $\eta^Y_1 = P_3: Y \to Y_1$ is s-bijective
\end{enumerate}
\end{lemma}
\begin{proof}
A generic element of $Y$ has the form $\left((y_0,\tilde{y}_2),x,(y'_0,x_1)   \right)$ where
\begin{equation}\nonumber
\rho^{YX}_0(y_0) = \pi_2(\tilde{\rho}^{YX}(\tilde{y}_2)), \rho^{YX}_0(y'_0) = \pi_1(x_1), y_0 = y_0', \tilde{\rho}^{YX}(\tilde{y}_2) = \eta_2(x), \text{ and } \eta_1(x) = x_1,
\end{equation}
which reduces to
\begin{equation}\label{formula2_54}
\pi_1(x_1) = \rho^{YX}_0(y_0) = \pi_2(\tilde{\rho}^{YX}(\tilde{y}_2)), \,\,\,\, \tilde{\rho}^{YX}(\tilde{y}_2) = \eta_2(x), \, \text{ and } \eta_1(x) = x_1.
\end{equation}
To prove (2), suppose that $\left((y_0,\tilde{y}_2),x,(y_0,x_1)   \right) \in Y$ and 
\[
\eta^Y_2\left((y_0,\tilde{y}_2),x,(y_0,x_1)   \right) = (y_0,\tilde{y}_2) \sim_u (y'_0,\tilde{y}'_2) \in Y_2,
\]
from which it follows that $\rho^{YX}_0(y'_0) = \pi_2(\tilde{\rho}^{YX}(\tilde{y}'_2))$.
Since $\pi_1(x_1) = \rho^{YX}_0(y_0) \sim_u \rho^{YX}_0(y'_0)$ and $\pi_1$ u-bijective, we have that there exists a unique $x'_1 \sim_u x$ such that $\pi_1(x'_1) = \rho^{YX}_0(y'_0)$. Similarly, since $\eta_2(x) = \tilde{\rho}^{YX}(\tilde{y}_2) \sim_u \tilde{\rho}^{YX}(\tilde{y}'_2)$ and $\eta_2$ u-bijective, we have that there exists a unique $x' \sim_u x$ such that $\eta_2(x') = \tilde{\rho}^{YX}(\tilde{y}'_2)$.

To complete the proof of (2), we need to show that $\left((y'_0,\tilde{y}'_2),x',(y'_0,x'_1)   \right) \in Y$. For this it suffices to show that $\eta_1(x') = x'_1$.  To this end, observe that 
\[
\pi_1(\eta_1(x')) = \pi_2(\eta_2(x')) = \pi_2(\tilde{\rho}^{YX}(\tilde{y}'_2)) = \rho^{YX}_0(y'_0) = \pi_1(x'_1) \quad \text{ and } \quad \eta_1(x') \sim_u \eta_1(x) = x_1 \stackrel{u}{\sim} x'_1,
\]
and since $\pi_1$ is u-bijective we have that $\eta_1(x') = x'_1$.

The proofs of (3) and (4) are analogous. Now (1) follows from (2) and the fact that $Y_2$ has totally disconnected unstable sets.
\end{proof}

\begin{lemma}\label{Lem4main}
Suppose $(Y,\psi)$, $(Y_0,\psi_0)$, $(Y_1,\psi_1)$, and $(Y_2,\psi_2)$ are as in Lemmas \ref{Lem1main}, \ref{Lem2main}, and \ref{Lem3main}. Then the following diagram commutes:
\begin{equation}\label{commutative_Y}
\begin{CD}
(Y,\psi) @>>\eta^Y_1> (Y_1,\psi_1) \\
@VV \eta^Y_2 V @VV \pi^Y_1 V \\
(Y_2,\psi_2) @>> \pi^Y_2 > (Y_0,\psi_0). \\
\end{CD}
\end{equation}
\end{lemma}

\begin{proof}
In the proof of Lemma \ref{Lem3main} we have shown that
\begin{align*}
\pi_1^Y \circ \eta_1^Y \left((y_0,\tilde{y}_2),x,(y_0,x_1)   \right) &= \pi_1^Y(y_0,x_1) = y_0 \quad \text{ and } \\ 
\pi_2^Y \circ \eta_2^Y \left((y_0,\tilde{y}_2),x,(y_0,x_1)   \right) &= \pi_2^Y(y_0,\tilde{y}_2) = y_0,
\end{align*}
proving the diagram commutes.
\end{proof}

Exchanging the roles of the s-bijective and u-bijective maps as needed in Lemmas \ref{Lem1main}, \ref{Lem2main}, \ref{Lem3main}, and \ref{Lem4main}, we adapt the proofs to obtain the following construction of a commutative diagram of Smale spaces. Let $(Z_0, \zeta_0)$ have totally disconnected stable sets such that $\rho^{ZX}_0: Z_0 \to X_0$ is u-bijective. Using $(Z_0, \zeta_0)$, we define Smale spaces
\begin{align*}
(Z_2, \zeta_2)&:= \fib{(Z_0, \zeta_0)}{\rho^{ZX}_0}{\pi_2}{(X_2, \varphi_2)} \quad \text{ and }\\
(Z_1, \zeta_1)&:= \fib{(Z_0, \zeta_0)}{\rho^{ZX}_0}{\pi_1 \circ \tilde{\rho}^{ZX}}{(\tilde{Z}_1, \tilde{\zeta}_1)}.
\end{align*}
These Smale spaces lead to the fibre product
\begin{equation}\nonumber
(Z, \zeta):=\{(z_2, x, z_1) \in \fib{(Z_2,\zeta_2)}{\rho^{ZX}_2}{\eta_2}{(X, \varphi)} \fib{}{\eta_1}{\rho^{ZX}_1}{(Z_1, \zeta_1)} \mid \pi^{Z}_2(z_2) = \pi^{Z}_1(z_1)\},
\end{equation}
and we obtain a commutative diagram
\begin{equation}\label{commutative_Z}
\begin{CD}
(Z,\zeta) @>>\eta^Z_1> (Z_1,\zeta_1) \\
@VV \eta^Z_2 V @VV \pi^Z_1 V \\
(Z_2,\zeta_2) @>> \pi^Z_2 > (Z_0,\zeta_0) \\
\end{CD}
\end{equation}
where $\pi_1^Z$ and $\eta_2^Z$ are u-bijective, and $\pi_2^Z$ and $\eta_1^Z$ are s-bijective.

A direct consequence of our constructions are the following s/u-bijective pairs:
\begin{align}
\rho &= (Y, \psi, \rho^{YX}, Z, \zeta, \rho^{ZX}) \quad \quad \,\, \text{ for $(X,\varphi)$}, \\
\rho_0 &= (Y_0, \psi_0, \rho^{YX}_0, Z_0, \zeta_0, \rho^{ZX}_0) \quad \text{ for $(X_0,\varphi_0)$}, \\
\rho_1 &= (Y_1, \psi_1, \rho^{YX}_1, Z_1, \zeta_1, \rho^{ZX}_1) \quad \text{ for $(X_1,\varphi_1)$}, \text{ and} \\
\rho_2 &= (Y_2, \psi_2, \rho^{YX}_2, Z_2, \zeta_2, \rho^{ZX}_2) \quad \text{ for $(X_2,\varphi_2)$}.
\end{align}
Combining these s/u-bijective pairs with the commutative diagrams \eqref{commutative_Y} and \eqref{commutative_Z} gives rise to the commutative cube of Smale spaces described below in \eqref{3d_diagram}.

\begin{lemma}
Suppose we have the following commutative diagram of Smale spaces
\begin{equation}\label{3d_diagram}
\begin{tikzpicture}[
back line/.style={densely dotted},
cross line/.style={preaction={draw=white, -,
line width=6pt}}]
\matrix (m) [matrix of math nodes,
row sep=2em, column sep=2em,
text height=1.5ex,
text depth=0.25ex]{
& (Y,\psi) & & (X,\varphi) & & (Z,\zeta) \\
(Y_2,\psi_2) & & (X_2,\varphi_2) & & (Z_2,\zeta_2) \\
& (Y_1,\psi_1) & & (X_1,\varphi_1) & & (Z_1,\zeta_1) \\
(Y_0,\psi_0) & & (X_0,\varphi_0) & & (Z_0,\zeta_0) \\
};
\path[->]
(m-1-2) edge node[above,pos=0.35] {$\rho^{YX}$} (m-1-4)
(m-1-2) edge [back line] node[left,pos=0.65] {$\eta^Y_1$} (m-3-2)
(m-1-4) edge [back line] node[left,pos=0.65] {$\eta_1$} (m-3-4)
(m-1-2) edge node[above,pos=0.65] {$\eta^Y_2$} (m-2-1)
(m-1-4) edge node[above,pos=0.65] {$\eta_2$} (m-2-3)
(m-2-1) edge node[above,pos=0.65] {$\rho^{YX}_2$} (m-2-3)
(m-2-1) edge node[left,pos=0.65] {$\pi^Y_2$} (m-4-1)
(m-3-2) edge [back line] node[above,pos=0.35] {$\rho^{YX}_1$} (m-3-4)
(m-3-2) edge [back line] node[above,pos=0.65] {$\pi^Y_1$} (m-4-1)
(m-4-1) edge node[above,pos=0.35] {$\rho^{YX}_0$} (m-4-3)
(m-3-4) edge [back line] node[above,pos=0.65] {$\pi_1$} (m-4-3)
(m-2-3) edge [cross line] node[left,pos=0.65] {$\pi_2$} (m-4-3)
(m-1-6) edge [cross line] node[above,pos=0.35] {$\rho^{ZX}$} (m-1-4) 
(m-1-6) edge [cross line] node[above,pos=0.65] {$\eta^Z_2$} (m-2-5)
(m-1-6) edge [cross line] node[left,pos=0.65] {$\eta^Z_1$} (m-3-6)
(m-3-6) edge [cross line] node[above,pos=0.65] {$\pi^Z_1$} (m-4-5)
(m-4-5) edge [cross line] node[above,pos=0.35] {$\rho^{ZX}_0$} (m-4-3)
(m-2-5) edge [cross line] node[above,pos=0.35] {$\rho^{ZX}_2$} (m-2-3)
(m-3-6) edge [back line] node[above,pos=0.35] {$\rho^{ZX}_1$} (m-3-4) 
(m-2-5) edge [cross line] node[left,pos=0.65] {$\pi^Z_2$} (m-4-5); 
\end{tikzpicture}
\end{equation}
where the spaces and maps are described in the previous lemmas. Then we have the following maps on homology:
$$
\eta_1^{s} \circ \eta_2^{s*}: H_N^s(X_2,\varphi_2) \to H_N^s(X_1,\varphi_1) 
$$
and
$$
\pi_1^{s*} \circ \pi_2^{s}: H_N^s(X_2,\varphi_2) \to H_N^s(X_1,\varphi_1). 
$$
\end{lemma}
\begin{proof}
Our constructions guarantee that all the relevant commutative diagrams satisfy the conditions of Theorem \ref{mod5.4.1}, which then gives the result.
\end{proof}

It remains to show that the two maps on homology in the previous lemma are equal.

\begin{lemma}
 For $L,M \geq 0$, define the SFTs $(\Sigma_{L,M}(\rho), \sigma)$, $(\Sigma_{L,M}(\rho_0), \sigma)$, $(\Sigma_{L,M}(\rho_1), \sigma)$, and $(\Sigma_{L,M}(\rho_2), \sigma)$ as in \cite[Defn 2.6.4]{put}. Then the maps defined via
\begin{eqnarray*}
(\eta^{\Sigma}_1)_{L,M} &=& (\eta_1^Y \circ P_1)_L \times (\eta_1^Z \circ P_2)_M \\
(\eta^{\Sigma}_2)_{L,M} &=& (\eta_2^Y \circ P_1)_L \times (\eta_2^Z \circ P_2)_M \\
(\pi^{\Sigma}_1)_{L,M} &=&  (\pi_1^Y \circ P_1)_L \times (\pi_1^Y \circ P_2)_M \\
(\pi^{\Sigma}_2)_{L,M} &=&  (\pi_2^Y \circ P_1)_L \times (\pi_2^Y \circ P_2)_M.\\
\end{eqnarray*}
are well-defined. Moreover, $(\pi^{\Sigma}_1)_{L,M}$ and $(\eta^{\Sigma}_2)_{L,M}$ are u-bijective, $(\pi^{\Sigma}_2)_{L,M}$ and $(\eta^{\Sigma}_1)_{L,M}$ are s-bijective, and
\[
\begin{CD}
(\Sigma_{L,M}(\rho), \sigma) @>>(\eta^{\Sigma}_1)_{L,M}> (\Sigma_{L,M}(\rho_1), \sigma) \\
@VV (\eta^{\Sigma}_2)_{L,M} V @VV (\pi^{\Sigma}_1)_{L,M} V \\
(\Sigma_{L,M}(\rho_2), \sigma) @>> (\pi^{\Sigma}_2)_{L,M} > (\Sigma_{L,M}(\rho_0), \sigma) \\
\end{CD}
\]
is a commutative diagram and 
\[
(\eta^{\Sigma}_2)_{L,M} \times (\eta^{\Sigma}_1)_{L,M}: (\Sigma_{L,M}(\rho), \sigma) \to \fib{(\Sigma_{L,M}(\rho_2), \sigma)}{(\pi^{\Sigma}_2)_{L,M}}{(\pi^{\Sigma}_1)_{L,M}}{(\Sigma_{L,M}(\rho_1), \sigma)}
\]
is a conjugacy.
\end{lemma}

\begin{proof}

We first prove that $(\eta^{\Sigma}_1)_{L,M}$ is well-defined and s-bijective. For simplicity of notation we prove the $M=L=0$ case, and note that the general case is similar.

To see the map is well-defined, note that a typical element of $\Sigma_{0,0}(\rho)$ has the form
\[
\left(((y_0,\tilde{y}_2),x,(y_0,x_1)),((z_0,x_2),x,(z_0,\tilde{z}_1)) \right)
\]
where
\begin{equation}\label{eq1:lem57}
\pi_1(x_1) = \rho^{YX}_0(y_0) = \pi_2(\tilde{\rho}^{YX}(\tilde{y}_2)), \ \ \tilde{\rho}^{YX}(\tilde{y}_2) = \eta_2(x),  \ \ \eta_1(x) = x_1.
\end{equation}
and
\begin{equation}\label{eq2:lem57}
\pi_2(x_2) = \rho^{ZX}_0(z_0) = \pi_1(\tilde{\rho}^{ZX}(\tilde{z}_1)), \ \ \tilde{\rho}^{ZX}(\tilde{z}_1) = \eta_1(x),  \ \ \eta_2(x) = x_2.
\end{equation}
The definition of $(\eta_1^\Sigma)_{0,0}$ implies that 
\[
(\eta^{\Sigma}_1)_{0,0}\left(((y_0,\tilde{y}_2),x,(y_0,x_1)),((z_0,x_2),x,(z_0,\tilde{z}_1)) \right) = ((y_0,x_1),(z_0,\tilde{z}_1)),
\]
We must show that $((y_0,x_1),(z_0,\tilde{z}_1)) \in \Sigma_{0,0}(\rho_1)$. That is, we must show that 
\[
\pi_1(x_1) = \rho^{YX}_0(y_0), \ \ \rho^{ZX}_0(z_0) = \pi_1(\tilde{\rho}^{ZX}(\tilde{z}_1)), \ \ x_1 = \tilde{\rho}^{ZX}(\tilde{z}_1).
\]
However, these equalities are immediate from \eqref{eq1:lem57} and \eqref{eq2:lem57}.

Next, we must show that $(\eta^{\Sigma}_1)_{(0,0)}$ is a factor map (i.e., that it commutes with dynamics and is surjective); we only give the details of the proof that it is surjective. Let $((y_0,x_1),(z_0,\tilde{z}_1))$ be an element in $\Sigma_{0,0}(\rho_1)$. Since $\eta_1$ and $\tilde{\rho}^{YX}$ are onto, there exists $x\in X$ and $\tilde{y}_2 \in \tilde{Y}_2$ such that $\eta_1(x)=x_1$ and $\tilde{\rho}^{YX}(\tilde{y}_2)=\eta_2(x)$. 

We must show that 
\[
\left(((y_0,\tilde{y}_2),x,(y_0,x_1)),((z_0,\eta_2(x)),x,(z_0,\tilde{z}_1)) \right) \in \Sigma_{0,0}(\rho).
\]
All of the required equalities (i.e.,  \eqref{eq1:lem57} and \eqref{eq2:lem57}) follow trivially or from a short computation. As a prototypical example, we show that $\pi_1(x_1) = \rho^{YX}_0(y_0) = \pi_2(\tilde{\rho}^{YX}(\tilde{y}_2))$; proofs of the other equalities are left to the reader. That $\pi_1(x_1)=\rho^{YX}_0(y_0)$ follows trivially from $((y_0,x_1),(z_0,\tilde{z}_1)) \in \Sigma_{0,0}(\rho_1)$. On the other hand, that $\pi_1(x_1)=  \pi_2(\tilde{\rho}^{YX}(\tilde{y}_2))$ follows from
\[ \pi_1(x_1)=(\pi_1\circ \eta_1)(x)=(\pi_2\circ \eta_2)(x)=\pi_2(\tilde{\rho}^{YX}(\tilde{y}_2)).\]

Next, we must show that $(\eta^{\Sigma}_1)_{(0,0)}$ is s-bijective (i.e., that, for each element in $\Sigma_{0,0}(\rho)$, $\left(((y_0,\tilde{y}_2),x,(y_0,x_1)),((z_0,x_2),x,(z_0,\tilde{z}_1))) \right)$ , the restriction of $(\eta^{\Sigma}_1)_{(0,0)}$ to 
\[ X^s\left(((y_0,\tilde{y}_2),x,(y_0,x_1)),((z_0,x_2),x,(z_0,\tilde{z}_1)) \right) \]
is a bijective map from 
\[ X^s\left(((y_0,\tilde{y}_2),x,(y_0,x_1)),((z_0,x_2),x,(z_0,\tilde{z}_1)) \right) \]
to 
\[ X^s\left((\eta^{\Sigma}_1)_{(0,0)}((y_0,\tilde{y}_2),x,(y_0,x_1)),((z_0,x_2),x,(z_0,\tilde{z}_1)) \right). \]

As the reader can verify, showing the map is injective is equivalent to showing that if 
\[\left(((y_0,\tilde{y}_2),x,(y_0,x_1)),((z_0,x_2),x,(z_0,\tilde{z}_1)) \right)  \stackrel{s}{\sim} \left(((y_0,\tilde{y}^{\prime}_2),x^{\prime},(y_0,x_1)),((z_0,x^{\prime}_2),x^{\prime},(z_0,\tilde{z}_1)) \right), 
\]
then they are equal. To show this latter statement, note that since $\eta_1$ is s-bijective and $\eta_1(x)=x_1=\eta_1(x^{\prime})$, we have that $x=x^{\prime}$. Then, $x_2^{\prime}=\eta_2(x^{\prime})=\eta_2(x)=x_2$. Finally, we must show that $\tilde{y}_2=\tilde{y}^{\prime}_2$. Using the fact that $\tilde{\rho}^{YX}$ is s-bijective, we need only show that $\tilde{\rho}^{YX}(\tilde{y}_2)=\tilde{\rho}^{YX}(\tilde{y}^{\prime}_2)$; however, this follows since $x=x^{\prime}$ and the equalities in  \eqref{eq1:lem57}. This completes the proof that the restriction is injective.

Finally, we show that the map (restricted to each stable equivalence class) is surjective. In other words, given 
\[ \left(((y_0,\tilde{y}_2),x,(y_0,x_1)),((z_0,x_2),x,(z_0,\tilde{z}_1)) \right) \in \Sigma_{0,0}(\rho) \]
and
\[
((y_0,x_1),(z_0,\tilde{z}_1)) \stackrel{s}{\sim} ((y'_0,x'_1),(z'_0,\tilde{z}'_1)) \in \Sigma_{0,0}(\rho_1)
\]
we must produce a preimage of  $((y'_0,x'_1),(z'_0,\tilde{z}'_1))$ which is in 
\[ X^s\left(((y_0,\tilde{y}_2),x,(y_0,x_1)),((z_0,x_2),x,(z_0,\tilde{z}_1)) \right). \]
Since $\eta_1$ is s-bijective and $\eta_1(x)=x_1 \stackrel{s}{\sim}x'_1$, there exists $x'\in X$ such that $x \stackrel{s}{\sim}x'$ and $\eta_1(x')=x'_1$. Similarly, since $\tilde{\rho}^{YX}$ is s-bijective and $\tilde{\rho}^{YX}(\tilde{y}_2)= \eta_2(x) \stackrel{s}{\sim} \eta_2(x')$, there exists $\tilde{y}'_2$ such that $\tilde{y}'_2 \stackrel{s}{\sim} \tilde{y}_2$ and $\tilde{\rho}^{YX}(\tilde{y}'_2)=\eta_2(x')$.
Finally, it is clear that 
\[\left(((y'_0,\tilde{y}'_2),x',(y'_0,x'_1)),((z'_0,\eta_2(x')),x',(z'_0,\tilde{z}'_1)) \right)\]
is a preimage provided that it is actually in $ \Sigma_{0,0}(\rho)$. However, this fact follows in the same way as in the proof that the map $(\eta^{\Sigma}_1)_{(0,0)}$ was surjective. 

The proofs that $(\pi^{\Sigma}_1)_{L,M}$ and $(\eta^{\Sigma}_2)_{L,M}$ are u-bijective, and that $(\pi^{\Sigma}_2)_{L,M}$ is s-bijective are analogous and are omitted. Furthermore, routine arguments show that the diagram commutes.

We now prove that $(\eta^{\Sigma}_2)_{L,M} \times (\eta^{\Sigma}_1)_{L,M}$ is surjective. A typical element of 
\[
\fib{(\Sigma_{L,M}(\rho_2), \sigma)}{(\pi^{\Sigma}_2)_{L,M}}{(\pi^{\Sigma}_1)_{L,M}} {(\Sigma_{L,M}(\rho_1), \sigma)}
\]
has the form
\begin{equation}\label{eq4:lem57}
\left((y_0^{(l)}, \tilde{y}_2^{(l)})_{l=0}^L, (z_0^{(m)},x_2^{(m)})_{m=0}^M , ((y^{(l)}_0)',x^{(l)}_1)_{l=0}^L, ((z^{(m)}_0)',\tilde{z}^{(m)}_1)_{m=0}^M \right),
\end{equation}
where $x_1^{(i)} = \tilde{\rho}^{Z}(\tilde{z}_1^{(j)})$ for all $i, j$, $x_2^{(i)} = \tilde{\rho}^{Y}(\tilde{y}_2^{(j)})$ for all $i, j$, $y_0^{(i)} = (y_0^{(i)})'$, $z_0 = (z_0^{(i)})'$ (as well as several other relations). So we can rewrite \eqref{eq4:lem57} as
\[
\left((y_0^{(l)}, \tilde{y}_2^{(l)})_{l=0}^L, (z_0^{(m)},x_2)_{m=0}^M , ((y^{(l)}_0),x_1)_{l=0}^L, ((z^{(m)}_0),\tilde{z}^{(m)}_1)_{m=0}^M \right).
\]
Since $\pi_2(\tilde{\rho}^Y(\tilde{y}_2^{(i)})) = \rho^{YX}_0(y_0^{(i)}) = \pi_1(x_1^{(i)})$, and 
\[
\eta_2 \times \eta_1: (X,\varphi) \to \fib{(X_2,\varphi_2)}{\pi_2}{\pi_1}{(X_1,\varphi_1)}
\]
is onto, there exists $x \in X$ such that for each $i$ we have $\eta_2(x) = \tilde{\rho}^Y(\tilde{y}_2^{(i)})$ and $\eta_1(x) = x_1$. Then we have
\[
\left( ((y^{(l)}_0,\tilde{y^{(l)}}_2),x,(y^{(l)}_0,x_1))_{l=0}^L , ((z^{(m)}_0,x_2),x,(z^{(m)}_0,\tilde{z^{(m)}}_1))_{m=0}^M \right) \in (\Sigma_{L,M}(\rho), \sigma),
\]
and this maps (under $(\eta^{\Sigma}_2)_{L,M} \times (\eta^{\Sigma}_1)_{L,M}$) to the generic point in \eqref{eq4:lem57}, as desired.

Finally, we show $(\eta^{\Sigma}_2)_{L,M} \times (\eta^{\Sigma}_1)_{L,M}$ is injective. Suppose
\[
\left( ((y^{(l)}_0,\tilde{y^{(l)}}_2),x,(y^{(l)}_0,x_1^{(l)}))_{l=0}^L , ((z^{(m)}_0,x_2^{(m)}),x,(z^{(m)}_0,\tilde{z^{(m)}}_1))_{m=0}^M \right) \in \Sigma_{L,M}(\rho)
\]
and
\[
\left( ((y^{(l)'}_0,\tilde{y^{(l)'}}_2),x',(y^{(l)'}_0,x_1^{(l)'}))_{l=0}^L , ((z^{(m)'}_0,x_2^{(m)'}),x',(z^{(m)'}_0,\tilde{z^{(m)'}}_1))_{m=0}^M \right) \in \Sigma_{L,M}(\rho)
\]
are two points that are not equal, but are equal after applying $\eta^{\Sigma}_2$.  We will show that they are not equal after applying $\eta^{\Sigma}_1$.
\begin{align*}
(\eta^{\Sigma}_1)_{L,M} \left( ((y^{(l)}_0,\tilde{y^{(l)}}_2),x,(y^{(l)}_0,x_1^{(l)}))_{l=0}^L , ((z^{(m)}_0,x_2^{(m)}),x \right.&\left.,(z^{(m)}_0,\tilde{z^{(m)}}_1))_{m=0}^M \right)\\
&=\left((y^{(l)}_0,x^{(l)}_1)_{l=0}^L, (z^{(m)}_0,\tilde{z^{(m)}}_1)_{m=o}^M \right) \\
(\eta^{\Sigma}_1)_{L,M}\left( ((y^{(l)'}_0,\tilde{y^{(l)'}}_2),x',(y^{(l)'}_0,x_1^{(l)'}))_{l=0}^L , ((z^{(m)'}_0,x_2^{(m)'}),x' \right.& \left. ,(z^{(m)'}_0,\tilde{z^{(m)'}}_1))_{m=0}^M \right) \\
&=\left((y^{(l)'}_0,x^{(l)'}_1)_{l=0}^L, (z^{(m)'}_0,\tilde{z^{(m)'}}_1)_{m=o}^M \right).
\end{align*}
Since these points are not equal to begin with, but are equal after applying $(\eta^{\Sigma}_2)_{L,M}$ we know that one of the following must hold.
\begin{enumerate}
\item $x^{(l)}_1 \neq x^{(l)'}_1$ for some $l$
\item $\tilde{z^{(m)}}_1 \neq \tilde{z^{(m)'}}_1$ for some $m$
\item $x \neq x'$
\end{enumerate}
If either of the first two hold, the result follows immediately.  If the third holds, then by hypothesis we have $\eta_2(x) = x_2 = x_2' = \eta_2(x')$, and it follows that $x^{(l)}_1 = \eta_1(x) \neq \eta_1(x') = x^{(l)'}_1$ and the result follows.
\end{proof}

We are now able to prove the main result of the section.

\begin{proof}[Proof of Theorem \ref{main_thm}]
By Theorem \ref{mod3.5.11}, for each $L,M$ we have
\[
(\eta_1^{\Sigma})^{s}_{L,M} \circ (\eta_2^{\Sigma})^{s*}_{L,M} = (\pi_1^{\Sigma})^{s*}_{L,M} \circ (\pi_2^{\Sigma})^{s}_{L,M}: D^s(\Sigma_{L,M}(\rho_2),\sigma) \to D^s(\Sigma_{L,M}(\rho_1),\sigma).
\]
Since 
\[
\eta_1^{s} \circ \eta_2^{s*}: H_N^s(X_2,\varphi_2) \to H_N^s(X_1,\varphi_1) 
\]
is defined by the composition $(\eta_1^{\Sigma})^{s}_{L,M} \circ (\eta_2^{\Sigma})^{s*}_{L,M}$ and
\[
\pi_1^{s*} \circ \pi_2^{s}: H_N^s(X_2,\varphi_2) \to H_N^s(X_1,\varphi_1). 
\]
is defined by the composition $(\pi_1^{\Sigma})^{s*}_{L,M} \circ (\pi_2^{\Sigma})^{s}_{L,M}$,
we can conclude that 
\[
\eta_1^{s} \circ \eta_2^{s*} = \pi_1^{s*} \circ \pi_2^{s}: H_N^s(X_2,\varphi_2) \to H_N^s(X_1,\varphi_1),
\]
the desired result.
\end{proof}

\section{A summary of our results for the unstable homology theory}\label{unstable_results}
In the preceding sections, only $H^s(X,\phi)$ was considered. The analogous results for $H^u(X,\phi)$ are stated in this section to facilitate easy referencing. The proofs are analogous to the $H^s(X,\phi)$ results and are omitted.

\begin{thm}
Suppose that 
\[
\begin{CD}
(\Sigma, \sigma) @>>\eta_1> (\Sigma_1, \sigma) \\
@VV \eta_2 V @VV \pi_1 V \\
(\Sigma_2, \sigma) @>> \pi_2 > (\Sigma_0, \sigma)\\
\end{CD}
\]
is a commutative diagram of non-wandering shifts of finite type in which $\eta_1$ and $\pi_2$ are s-bijective factor maps, and $\eta_2$ and $\pi_1$ are u-bijective factor maps. Moreover, suppose $\eta_2 \times \eta_1 : (\Sigma, \sigma) \to \fib{(\Sigma_2, \sigma)}{\pi_2}{\pi_1}{(\Sigma_1, \sigma)}$ is a conjugacy.
Then
\[
\eta_2^{u} \circ \eta_1^{u*}= \pi_2^{u*} \circ \pi_1^{u} : H_N^u(\Sigma_1,\sigma) \to H_N^u(\Sigma_2,\sigma).
\]
\end{thm}

\begin{thm}\label{prelim_u_naturality}
Let $\pi = (Y,\psi,\pi_s,Z,\zeta,\pi_u)$ and $\pi' = (Y',\psi',\pi'_s,Z',\zeta',\pi'_u)$ be s/u-bijective pairs for the non-wandering Smale spaces $(X,\phi)$ and $(X',\phi')$ respectively.  Let $\eta = (\eta_X,\eta_Y,\eta_Z)$ be a triple of s-bijective factor maps, or a triple of u-bijective factor maps.
\begin{align*}
\eta_X&: (X,\phi) \to (X',\phi') \\
\eta_Y&: (Y,\psi) \to (Y',\psi') \\
\eta_Z&: (Z,\zeta) \to (Z',\zeta')
\end{align*}
such that the following diagram commutes.
\[
\begin{CD}
(Y,\psi) @>>\pi_s> (X,\varphi) @<<\pi_u< (Z,\zeta)\\
@VV \eta_Y V @VV \eta_X V @VV \eta_Z V\\
(Y',\psi') @>> \pi'_s > (X',\varphi') @<<\pi'_u< (Z',\zeta').\\
\end{CD}
\]
If $\eta$ is s-bijective, we assume $\pi_u \times \eta_Z: (Z, \zeta) \to \fib{(X, \phi)}{\eta_X}{\pi_u'}{(Z', \zeta')}$ is a conjugacy.
Then, for each $L,M \geq 0$ 
\[
\begin{CD}
\Sigma_{L,M}(\pi) @>>\eta_{L,M}> \Sigma_{L,M}(\pi') \\
@VV \delta_{,m} V @VV \delta_{,m} V \\
\Sigma_{L,M-1}(\pi) @>> \eta_{L,M-1} > \Sigma_{L,M-1}(\pi') \\
\end{CD}
\]
is a commutative diagram and $\eta_{L,M} \times \delta_{,m}: \Sigma_{L,M}(\pi) \to \fib{\Sigma_{L,M}(\pi')}{\delta'_{'m}}{\eta_{L,M-1}}{\Sigma_{L,M-1}(\pi)}$ is a conjugacy.
In particular, $\eta$ induces a chain map on the double complexes used to define $H^u(\pi)$ and $H^u(\pi')$. Similarly, if $\eta$ is u-bijective, we assume $\pi_s \times \eta_Y: (Y, \psi) \to \fib{(X, \phi)}{\eta_X}{\pi_s'}{(Y', \psi')}$ is a conjugacy.
Then, for each $L,M \geq 0$ 
\[
\begin{CD}
\Sigma_{L,M}(\pi) @>>\delta_{l,} > \Sigma_{L-1,M}(\pi) \\
@VV \eta_{L,M} V @VV \eta_{L-1,M} V \\
\Sigma_{L,M}(\pi') @>> \delta_{l,} > \Sigma_{L-1,M}(\pi') \\
\end{CD}
\]
is a commutative diagram and $\delta_{l,} \times \eta_{L,M}: \Sigma_{L,M}(\pi) \to \fib{\Sigma_{L-1,M}(\pi)}{\eta_{L-1,M}}{\delta_{l,}}{\Sigma_{L,M}(\pi')}$ is a conjugacy.
In particular, $\eta$ induces a chain map on the double complexes used to define $H^u(\pi)$ and $H^u(\pi')$.
\end{thm}

\begin{thm}\label{Hu_naturality}
Let $\eta_X:(X,\phi) \to (X',\phi')$ be a factor map between non-wandering Smale spaces.  Consider the following diagrams 
\[
\begin{CD}
(Y,\psi) @>>\pi_s> (X,\varphi) @<<\pi_u< (Z,\zeta)\\
@VV \eta_Y V @VV \eta_X V @VV \eta_Z V\\
(Y',\psi') @>> \pi'_s > (X',\varphi') @<<\pi'_u< (Z',\zeta')\\
\end{CD}
\ \ \textrm{and} \ \ 
\begin{CD}
(\tilde{Y},\tilde{\psi}) @>>\tilde{\pi}_s> (X,\varphi) @<<\tilde{\pi}_u< (\tilde{Z},\tilde{\zeta})\\
@VV \tilde{\eta}_Y V @VV \eta_X V @VV \tilde{\eta}_Z V\\
(\tilde{Y}',\tilde{\psi}') @>> \tilde{\pi}'_s > (X',\varphi') @<<\tilde{\pi}'_u< (\tilde{Z}',\tilde{\zeta}')\\
\end{CD}
\]
\begin{enumerate}
\item If $\eta_X$ is u-bijective, and the above diagrams represent two different sets of data which both satisfy the hypotheses of Theorem \ref{prelim_u_naturality}, then
\[
\Theta' \circ \eta^u = \tilde{\eta}^u \circ \Theta
\]
where $\eta^u$ is defined via the the first diagram (via $\eta = (\eta_X,\eta_Y,\eta_Z$)), $\tilde{\eta}^u$ is defined via the second diagram (via $\tilde{\eta} = (\eta_X,\tilde{\eta_Y},\tilde{\eta_Z}$)), and $\Theta': H^u(\pi') \to H^u(\tilde{\pi}')$ and $\Theta: H^u(\pi) \to H^u(\tilde{\pi})$ are the isomorphisms described in the proof of \cite[Theorem 5.5.1]{put}.
\item If $\eta_X$ is s-bijective, and the above diagrams represent two different sets of data which both satisfy the hypotheses of Theorem \ref{naturality}, then
\[
\Theta \circ \eta^{u*} = \tilde{\eta}^{u*} \circ \Theta'
\]
where $\eta^{u*}$ is defined via the the first diagram (via $\eta = (\eta_X,\eta_Y,\eta_Z$)), $\tilde{\eta}^{u*}$ is defined via the second diagram (via $\tilde{\eta} = (\eta_X,\tilde{\eta_Y},\tilde{\eta_Z}$)), and $\Theta': H^u(\pi') \to H^u(\tilde{\pi}')$ and $\Theta: H^u(\pi) \to H^u(\tilde{\pi})$ are the isomorphisms described in the proof of \cite[Theorem 5.5.1]{put}.
\end{enumerate}
\end{thm}

\begin{thm}\label{main_thm_u}
Suppose 
\[
\begin{CD}
(X,\varphi) @>>\eta_1> (X_1,\varphi_1) \\
@VV \eta_2 V @VV \pi_1 V \\
(X_2,\varphi_2) @>> \pi_2 > (X_0,\varphi_0) \\
\end{CD}
\]
is a commutative diagram of non-wandering Smale spaces in which $\eta_1$ and $\pi_2$ are s-bijective factor maps, and $\eta_2$ and $\pi_1$ are u-bijective factor maps. Moreover, suppose $\eta_2 \times \eta_1 : (X, \varphi) \to \fib{(X_2, \varphi_2)}{\pi_2}{\pi_1} {(X_1, \varphi_1)}$ is a conjugacy.
Then
\[
\eta_2^{u} \circ \eta_1^{u*} = \pi_2^{u*} \circ \pi_1^{u}: H_N^u(X_1,\varphi_1) \to H_N^u(X_2,\varphi_2). 
\]
\end{thm}

\section{Examples}\label{examples}

In this section we provide examples showing the extent to which the conclusions of the theorems in this paper can fail without the hypotheses we have imposed.

Our first example shows that it is necessary that the map $\eta_2 \times \eta_1$ in Theorem \ref{mod3.5.11} is a conjugacy.

\begin{example}\label{2to1example}
Let $G$ and $H$ be the following directed graphs
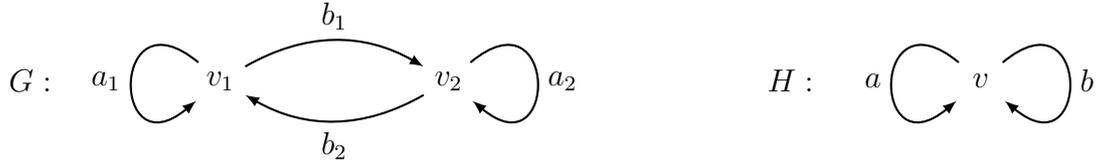
\begin{center}
\begin{tikzpicture}
\node at (-2.5,0) {$G:$};
\node at (0,0) {$v_1$};
\node[vertex] (vertexe) at (0,0)   {$\quad$}
	edge [->,>=latex,out=140,in=220,loop,thick] node[left,pos=0.5]{$a_1$} (vertexg);
\node at (3,0) {$v_2$};
\node[vertex] (vertexa) at (3,0)  {$\quad$}
	edge [<-,>=latex,out=150,in=30,thick] node[auto,swap,pos=0.5]{$b_1$} (vertexe)
	edge [->,>=latex,out=210,in=330,thick] node[auto,pos=0.5]{$b_2$} (vertexe)
	edge [->,>=latex,out=40,in=320,loop,thick] node[auto,pos=0.5]{$a_2$} (vertexg);
\node at (7.5,0) {$H:$};
\node at (10,0) {$v$};
\node[vertex] (vertexv) at (10,0)   {$\quad$}
	edge [->,>=latex,out=40,in=320,loop,thick] node[auto,pos=0.5]{$b$} (vertexg)
	edge [->,>=latex,out=140,in=220,loop,thick] node[left,pos=0.5]{$a$} (vertexg);
\end{tikzpicture}
\end{center}
Suppose $\pi$ is the graph homomorphism which sends $v_1,v_2 \mapsto v$, $a_1,a_2 \mapsto a$, and $b_1,b_2 \mapsto b$.  Routine computations show that $\pi$ is both left-covering and right-covering, see \cite[Definition 2.5.16]{put}.  Abusing notation we also let $\pi$ denote the factor map $\pi: (\Sigma_G,\sigma) \to (\Sigma_H, \sigma)$.  It follows from \cite[Theorem 2.5.17]{put} that $\pi$ is both s-bijective and u-bijective. Thus, the diagram
\[
\begin{CD}
(\Sigma_G, \sigma) @>>\pi> (\Sigma_H, \sigma) \\
@VV \pi V @VV id V \\
(\Sigma_H, \sigma) @>> id > (\Sigma_H, \sigma)\\
\end{CD}
\]
satisfies all the hypotheses of Theorem \ref{mod3.5.11} except that $\pi \times \pi$ is not 1-to-1; it is 2-to-1. In this case $\pi^s \circ \pi^{s*}: D^s(\Sigma_H,\sigma) \to D^s(\Sigma_H,\sigma)$ is multiplication by 2, while $id^{s*}\circ id^s$ is the identity.

On the other hand, the diagram
\[
\begin{CD}
(\Sigma_G, \sigma) @>>id> (\Sigma_G, \sigma) \\
@VV id V @VV \pi V \\
(\Sigma_G, \sigma) @>> \pi > (\Sigma_H, \sigma)\\
\end{CD}
\]
satisfies all of the hypotheses of Theorem \ref{mod3.5.11} \emph{except} that $\pi \times \pi$ is not onto. In this case $id^s \circ id^{s*}:  D^s(\Sigma_G,\sigma) \to D^s(\Sigma_G,\sigma)$ is the identity, while $\pi^{s*} \circ \pi^s$ is multiplication by 2.
\end{example}

In the next example we modify the previous example slightly to show what can go wrong in Theorem \ref{naturality} if the product map is not a conjugacy.

\begin{example}
Let $G$, $H$, and $\pi$ be as in Example \ref{2to1example}, and consider the following commutative diagram.
\[
\begin{CD}
(\Sigma_H,\sigma) @>> id > (\Sigma_H,\sigma) @<< \pi < (\Sigma_G,\sigma)\\
@VV id V @VV id V @VV \pi V\\
(\Sigma_H,\sigma) @>> id > (\Sigma_H,\sigma) @<< id < (\Sigma_H,\sigma).\\
\end{CD}
\]
In the notation of Theorem \ref{naturality}, we have s/u-bijective pairs $\pi = (\Sigma_H,\sigma, id, \Sigma_G, \sigma, \pi)$ and $\pi' = (\Sigma_H,\sigma, id, \Sigma_H, \sigma, id)$, and the triple of s-bijective maps $\eta = (id, id, \pi)$. This data satisfies all the hypotheses of Theorem \ref{naturality} \emph{except} that $\pi \times \pi$ is not 1-to-1.
It is then straightforward to verify that in the diagram
\[
\begin{CD}
\Sigma_{0,1}(\pi) @>>\eta_{0,1}> \Sigma_{0,1}(\pi') \\
@VV \delta_{,1} V @VV \delta_{,1} V \\
\Sigma_{0,0}(\pi) @>> \eta_{0,0} > \Sigma_{0,0}(\pi') \\
\end{CD}
\]
the map $\delta_{'1} \times \eta_{0,1}$ is not 1-to-1, and hence $\eta^s: H^s(\pi) \to H^s(\pi')$ is not well defined. 
\end{example}

The previous example also lead to an example showing that, in considering the naturality of the Homology theory, one must be careful in the choice of s/u-bijective pairs (as in Theorem \ref{main_naturality}).

\begin{example}
Let 
\[
\begin{CD}
(\Sigma_G,\sigma) @>> id > (\Sigma_G,\sigma) @<< id < (\Sigma_G,\sigma)\\
@VV \pi V @VV \pi V @VV id V\\
(\Sigma_H,\sigma) @>> id > (\Sigma_H,\sigma) @<<\pi< (\Sigma_G,\sigma)\\
\end{CD}
\ \ \textrm{and} \ \ 
\begin{CD}
(\Sigma_G,\sigma) @>> id > (\Sigma_G,\sigma) @<< id < (\Sigma_G,\sigma)\\
@VV \pi V @VV \pi V @VV \pi V\\
(\Sigma_H,\sigma) @>> id > (\Sigma_H,\sigma) @<< id < (\Sigma_H,\sigma)\\
\end{CD}
\]
be two sets of data as in the statement of Theorem \ref{main_naturality}, but notice that the hypotheses are not fully satisfied, for example, $id \times id: (\Sigma_G,\sigma) \to \fib{(\Sigma_G,\sigma)}{\pi}{\pi}{(\Sigma_G,\sigma)}$ is not onto.

In this case, the ``isomorphism'' $\Theta'$ which appears in Theorem \ref{main_naturality} is the ``map'' on homology in the previous example, and is thus not well defined.
\end{example}


\begin{thebibliography}{99}

\bibitem{Bow} R. Bowen, {\em On axiom A diffeomorphisms},
AMS-CBMS Reg. Conf. {\bf 35}, Providence, 1978.

\bibitem{Fri} D. Fried, {\em Finitely presented dynamical systems} , Ergodic Th. and Dynam. Sys. {\bf 7} (1987), 489--507.

\bibitem{Kri} W. Krieger, {\em On dimension functions and topological Markov chains}, Invent. Math. {\bf 56} (1980), 239--250.

\bibitem{LM} D. Lind and B. Marcus, {\em An Introduction to Symbolic Dynamics and Coding}, Revised Printing, Cambridge Univ. Press, Cambridge, 1999.

\bibitem{put}I. F. Putnam. {\it A homology theory for Smale spaces}, Memoirs of the A.M.S. {\bf 232}, Providence, 2014.

\bibitem{Rue1} D. Ruelle, {\em Thermodynamic formalism}, Second Ed., Cambridge Univ. Press, Cambridge, 2004.

\bibitem{Sma} S. Smale, {\em Differentiable dynamical systems}, Bull. A.M.S. {\bf 73} (1967), 747--817.

\end{thebibliography}
\end{document}